    \documentclass[11pt, twoside]{article}
\pdfoutput=1

  \usepackage{amsmath,amsthm,amssymb}
  \usepackage{times}
  \usepackage{enumerate}
  \usepackage{authblk}
  \usepackage{graphicx,color}
  \usepackage{hyperref}
  \usepackage{dsfont}
  \usepackage{float}
  \usepackage{caption}
  \usepackage{subcaption}
  
  \usepackage{ulem}

  \newcommand{\bproof}{\begin{proof}}
  	\newcommand{\eproof}{\end{proof}}
  \newcommand{\ve}{{\varepsilon}}

  \newcommand{\ben}{\begin{equation}}
  \newcommand{\een}{\end{equation}}

  \newcommand{\benn}{\begin{equation*}}
  \newcommand{\eenn}{\end{equation*}}

  \newcommand{\R}{\mathbf{R}}

  \newcommand{\divv}{\operatorname{div}}

  \newcommand{\varbf}{\mathbf{{\varphi}}}
  \newcommand{\ub}{\mathbf{u}}

  \newcommand{\pb}{\mathbf{p}}

  \newcommand{\Eb}{\textbf{E}}

  \newcommand{\psib}{\mathbf{{\psi}}}

  \newcommand{\dt}{{\frac{d}{dt}}}

  \newcommand{\tr}{\operatorname{tr}}

  \newcommand{\VV}{\theta}

  \newcommand{\une}{u_n}
  \newcommand{\udi}{u_d}
  \newcommand{\pdi}{p_d}
  \newcommand{\pne}{p_n}
  \newcommand{\phdi}{\varphi_d}
  \newcommand{\phne}{\varphi_n}
  \newcommand{\psdi}{\psi_d}
  \newcommand{\psne}{\psi_n}
  
  \newtheorem{thmx}{Assumption}
  
  \newtheorem{thmy}{Assumption}

  \newcommand{\ds}{\, ds}
  \newcommand{\dx}{\, dx}
  
  \newcommand{\dn}{\partial_n}
  
  \newcommand{\di}{\operatorname{div}}

  \newcommand{\curl}{\text{curl}}

  \newcommand{\Sb}{\mathbf{S}}
  \newcommand{\Sf}{\mathfrak{S}}
  
  \newcommand{\bil}{\mathcal B}

    \theoremstyle{definition}
    \newtheorem{thrm}{Theorem}[section]
    
  \newtheorem{theorem}[thrm]{Theorem}
  \newtheorem{lemma}[thrm]{Lemma}
  \newtheorem{corollary}[thrm]{Corollary}
  \newtheorem{remark}[thrm]{Remark}
  
  \newtheorem{definition}[thrm]{Definition}
  
  \newtheorem{assumption}[thrm]{Assumption}
  
  \newtheorem{proposition}[thrm]{Proposition}

  \numberwithin{equation}{section}

  \newcommand{\subjclass}[1]{\bigskip\noindent\emph{2010 Mathematics Subject Classification:}\enspace#1}
  \newcommand{\keywords}[1]{\noindent\emph{Keywords:}\enspace#1}

 \usepackage{epsfig} 

 \usepackage[english]{babel}

\usepackage{fancyhdr}
\usepackage{cite}
\usepackage{graphicx,color}
\usepackage{enumitem}

\usepackage{srcltx}

\setenumerate[0]{label=(\alph*)}

\usepackage{hyperref}
\hypersetup{%
    colorlinks=True, 
    citecolor=blue,
    linkcolor=black}

\begin{document}

\title{Distributed shape derivative via averaged adjoint method and applications}

\author[1]{Antoine Laurain\thanks{Universidade de S\~{a}o Paulo, Instituto de Matem\'atica e Estat\'istica, Departamento de Matem\'atica Aplicada, Rua do Mat\~{a}o, 1010
		Cidade Universitaria, CEP 05508-090, S\~{a}o Paulo, SP, Brazil}}
\author[2]{Kevin Sturm\thanks{Universit\"at Duisburg-Essen, Fakult\"at f\"ur Mathematik,
		Thea-Leymann-Stra\ss e 9,  D-45127 Essen, Germany;  {\bf email}: kevin.sturm@uni-due.de}}
\affil[1]{Department of Mathematics, University of S\~{a}o Paulo}
\affil[2]{Faculty of Mathematics, University of Essen-Duisburg}

\maketitle

\begin{abstract} 
	The structure theorem of Hadamard-Zol\'esio states that the derivative of a shape functional is a distribution on the boundary of the domain depending only on the normal perturbations of a smooth enough boundary. 
	Actually the domain representation, also known as distributed shape derivative, is more general than the boundary expression as it is well-defined for shapes having a lower regularity.
	It is customary in the shape optimization literature  to assume regularity of the domains and use the boundary expression of the shape derivative for numerical algorithms. 
	In this paper we describe several advantages of the distributed shape derivative in terms of generality, easiness of computation and numerical implementation. 
	We identify a tensor  representation of the distributed shape derivative, study its properties and show how it allows to recover the boundary expression directly. 
	We  use a novel Lagrangian approach, which is applicable to a large class of shape optimization problems, to compute the distributed shape derivative. 
	We also apply the technique to retrieve the distributed shape derivative for  electrical impedance tomography.  
	Finally we explain how to adapt the level set method to the distributed shape derivative framework and present numerical results. 
\end{abstract}
\newline
\maketitle
\section*{Introduction}
In his research on elastic plates \cite{Ha} in 1907, Hadamard 
showed how to obtain the derivative of a shape functional $J(\Omega)$ 
by considering normal perturbations of the  boundary $\partial\Omega$ of a smooth set  $\Omega$. This fundamental result of shape optimization  was  made rigorous later by Zol\'esio \cite{MR2731611} in the so-called ``structure theorem''. When $J(\Omega)$  and the domain are smooth enough, one may also write the shape derivative as an integral over $\partial\Omega$, which is the canonical form in the shape optimization literature. 

However, when $\Omega$ is less regular, the shape derivative can  often be written as a domain integral even when the boundary expression is not available. The {\it domain expression} also known as {\it distributed shape derivative} has been generally ignored in the shape optimization literature for several reasons:  firstly the boundary representation provides a straightforward way of determining an explicit descent direction since it depends linearly on the boundary perturbation $\VV$ and not on its gradient, secondly this descent direction only needs to be defined on the boundary. When considering the domain expression, these two advantages disappear as the shape derivative is defined on $\Omega$ and depends on the gradient of $\VV$, so that a partial differential equation needs to be solved to obtain a descent direction $\VV$ on  $\Omega$. 

It seems that these drawbacks  would definitely rule out the distributed shape derivative, however they turn out to be less dramatic than expected in many situations and the domain formulation has other less foreseeable advantages  over the boundary representation. In this paper we advocate for the use of the distributed shape derivative and discuss the advantages of this formulation. 

The boundary representation has the following drawbacks. First of all  if the data is not smooth enough the integral representation does not exist so that the more general domain representation is the only rigorous alternative. Even when the boundary representation exists and has the form
$\int_{\partial\Omega} g\, \VV\cdot n $, it is usually not legitimate to choose $\VV\cdot n = -g$ on $\partial\Omega$ for a descent direction if $g$ is not smooth enough, for instance if $g\in L^1(\partial\Omega)$. Therefore, a smoother $\VV$ must be chosen, which requires to solve a partial differential equation on the boundary $\partial\Omega$. When taking $\VV\cdot n = -g$ is legitimate, it might still not be desirable as this may yield a $\VV$ with  low regularity, in which case one needs to regularize $\VV$ on the boundary as well. 
In these cases the first advantage of the boundary representation disappears. The second advantage of the boundary 
representation is that the perturbation field only needs to be defined on the boundary instead of on the whole domain, reducing the cost of the computation. Actually, the distributed shape derivative also has its support on the boundary, and may be computed in a small neighborhood of the boundary so that the additional cost is minimal. 
In addition, in most shape optimization applications, $g$ is the restriction of a function defined in a neighborhood of the boundary and not a quantity depending only on the boundary such as the curvature. Therefore from a practical point of view, $g$ must be evaluated in a neighborhood of $\partial\Omega$ anyway. 
Also, in many numerical applications, $\VV$ must be extended to a neighborhood of $\Gamma$ or even to the entire domain $\Omega$. This is the case for level set methods for instance, where the level set function must be updated on $\Omega$, or when one wishes to update the mesh along with the domain update, to avoid re-meshing the new domain. The distributed shape derivative then directly gives an extension of $\VV$ well-suited to the optimization problem. 

Recent results have shown that the distributed shape derivative is also more accurate than the boundary representation from a numerical point of view; see \cite{MR3348199} for a comparison. Indeed functions such as gradients of the state and adjoint state appearing in the distributed shape derivative only need to be defined at grid points and not on the interface. Therefore one avoids interpolation of these irregular terms.   This is particularly useful for transmission problems where the boundary representation requires to compute the jump of a function over the interface, a delicate and error-prone operation from the numerical point of view. 

Having considered these equivalent expressions of the shape derivative (i.e. boundary and domain expression) leads to a general form of the shape derivative using tensors. We introduce such a tensor representation in Section  \ref{sec:tensor_representation} which covers a large class of problems and in particular contains the boundary and domain expression. We show how this abstract form allows to identify simple relations between the domain and boundary expressions of the shape derivative.  

In this paper we also extend and simplify the averaged adjoint method from \cite{sturm13}, a  Lagrangian-type method which is well-suited to compute the shape derivative of a cost function in an efficient way.
Lagrangian methods are commonly used in shape optimization and have the advantage of providing the shape derivative without the need to compute the material derivative of the state; see \cite{MR862783,MR948649,MR2166150,sturm13,MeftBelh13,SturmHeomHint13}.    
Compared to these known shape-Lagrangian methods, the averaged adjoint method is fairly general due to minimal required conditions. The assumptions are for instance less restrictive than those required for the theorem of Correa-Seeger \cite{MR948649}, therefore it can be applied to more general situations such as non-convex functionals.  As the direct approach our method can also be applied for problems depending on nonlinear partial differential equations. In this paper we give an example of application to a transmission problem (in electrical impedance tomography - see Section \ref{section4}).  Our method  provides the domain expression of the shape derivative and the boundary expression can be computed easily from the tensor representation of the domain expression.

To complete the numerical implementation aspect, we also show  how the domain expression of the shape derivative can be used in the level set method framework \cite{MR2033390,FLS,MR965860,MR2806573,HLsegmentation,MR2472382,MR2459656,MR2356899}. 
The level set method can be modified to use the domain expression which leads to a method which is actually easier to implement.
Combining all these techniques, we obtain a straightforward and general  way of solving the shape optimization problem, from the rigorous theoretical computation of the shape derivative to the numerical implementation.

In Section \ref{section1} we recall the concept of shape derivative and the structure theorem on an abstract level. 
In Section \ref{section2} a shape-Lagrangian method, the averaged adjoint method, is described. 
In Section \ref{sec:tensor_representation} we identify a general tensor representation of the shape derivative, establish some of its properties, and give a few examples.
In Section \ref{section3} we explain how to compute descent directions for the distributed shape derivative for use in gradient methods. 
In Section \ref{section4} we apply the results of Sections \ref{section2} and  \ref{sec:tensor_representation} to the inverse problem of electrical impedance tomography. 
In Section \ref{section5} we extend the level set method to the case of the distributed shape derivative and finally  in Section \ref{section6} we show numerical results for various problems including the  problem of electrical impedance tomography.

\section{The structure theorem revisited}\label{section1}

Our aim in this section is to describe properties of the shape derivative on an abstract level and to emphasize that all representations of the shape derivative satisfy the same {\it structure theorem}.

Let $\mathcal{P}(D)$ be the set of subsets of $D\subset\R^d$ compactly contained in $D$, where the so-called ``universe''
$D\subset \R^d$ is assumed to be open and bounded.
Define for $k\geq 0$ and $0\leq \alpha \leq 1$,
\begin{align}
C^{k,\alpha}_c(D,\R^d) &:=\{\VV\in C^{k,\alpha}(D,\R^d)|\quad \VV\text{ has compact support in } D\}.
\end{align}
Also for given domain $\Omega\subset D$ with at least a $C^1$ boundary we introduce the space of vector field
\begin{align}
C^{k,\alpha}_{\partial \Omega}(D ,\R^d) &:=\{\VV\in C^{k,\alpha}_c( D,\R^d)|\quad \VV\cdot n =0 \text{ on } \partial \Omega\}
\end{align}
where $n$ is the outward unit normal vector to $\Omega$.

Consider a vector field $\VV\in C^{0,1}_c(D,\R^d)$ and the associated flow
$\Phi_t^{\VV}:\overline{D}\rightarrow \R^d$, $t\in [0,\tau]$ defined for each $x_0\in \overline{D}$ as $\Phi_t^{\VV}(x_0):=x(t)$, where $x:[0,\tau]\rightarrow \R$ solves 
\begin{align}\label{Vxt}
\begin{split}
\dot{x}(t)&= \VV(x(t))    \quad \text{ for } t\in (0,\tau),\quad  x(0) =x_0.
\end{split}
\end{align}
We will sometimes use the simpler notation $\Phi_t=\Phi_t^{\VV}$ when no confusion is possible.
Since $\VV\in C^{0,1}_c(D,\R^d)$ we have by Nagumo's theorem \cite{MR0015180} that for fixed $t\in [0,\tau]$ the flow $\Phi_t$ is a homeomorphism from $D$ into itself and maps boundary onto boundary and
interior onto interior. Further, we consider the family 
\begin{equation}\label{domain}
\Omega_t := \Phi_t^{\VV}(\Omega) 
\end{equation}
of perturbed domains. 

In the following let $J : \mathfrak P \rightarrow \R$ be a shape function defined on
some admissible set $\mathfrak P\subset \mathcal P(D)$.

\begin{definition}\label{def1} 
	The Eulerian semiderivative of $J$ at $\Omega$ in direction $\theta \in C^{0,1}_c(D,\R^d)$, when the limit exists,
	is defined by
	\ben
	dJ(\Omega)(\VV):= \lim_{t \searrow 0}\frac{J(\Omega_t)-J(\Omega)}{t}.
	\een
	\begin{itemize}
		\item[(i)] $J$ is said to be \textit{shape differentiable} at $\Omega$ if it has a Eulerian semiderivative at $\Omega$ for all $\VV \in C^\infty_c(D,\R^d)$ and the mapping
		\begin{align*}
		dJ(\Omega):  C^\infty_c(D,\R^d)  &  \to \R,\; \VV     \mapsto dJ(\Omega)(\VV)
		\end{align*}
		is linear and continuous, in which case $dJ(\Omega)(\VV)$ is called the \textit{shape derivative} at $\Omega$.
		\item[(ii)] 
		The shape derivative $dJ(\Omega)$ is of finite order if there is an integer $l\ge 0$ and a constant $c>0$ such that for each compact $K\subset D$ 
		$$|dJ(\Omega)(\theta)|  \le c \|\theta\|_{l} \quad \forall \theta\in C^\infty_c(K,\R^d),$$
		where $\|\theta\|_{l} := \sum_{|\alpha| \le l} |D^\alpha \theta|_\infty$.
		The smallest such integer $l\ge 0$ is called order of $dJ(\Omega)$.
	\end{itemize}
\end{definition}

The shape derivative from Definition \ref{def1} has a particular structure. 
Intuitively, it is clear that the form functional stays constant 
for a transformation $\Phi$ that leaves $\Omega$ unchanged, that is $\Phi(\Omega)=\Omega$,  even if some points inside $\Omega$ move  and consequently the shape derivative is zero in this case. This property is valid  when $\Omega$ is open or closed; cf. \cite{MR2731611}. Mathematically, this is expressed in the following basic theorem 
proved in \cite{zolesio1979identification}.

\begin{theorem}\label{lem:1}
	Let  $\Omega \in \mathfrak P$ be open or closed. Let $\VV \in C^{0,1}_c(D,\R^d)$ be a vector field with 
	compact support in $\Omega$ and denote by
	$\Phi_t$ its flow defined in \eqref{Vxt}.  
	Then we have 
	$$  dJ(\Omega)(\VV)=0. $$
\end{theorem}
%

Note that the shape derivative of $J(\Omega)$ always exists for vector fields with compact support in $\Omega$, even if it does not exist for other vector fields. An important consequence of Theorem \ref{lem:1}, also for numerical methods, is that {\it independently of the representation} of the shape derivative and the regularity of the domain $\Omega$, the values of $\VV$ outside the boundary of $\Omega$ have no influence on the shape derivative. 

\begin{corollary}
	Let $\Omega \in \mathfrak P$ be a set with $C^{1}$-boundary. 
	Assume that $J$ is shape differentiable on $\mathfrak P$.
	Let $\VV \in C^{0,1}_{\partial\Omega}(D,\R^d)$.
	Then we have
	$$ dJ(\Omega)(\VV)=0. $$
\end{corollary}

The previous discussion immediately yields the following fundamental result of shape optimization.
\begin{theorem}[Structure Theorem]\label{thm:structure_theorem}
	Assume $\Gamma :=\partial \Omega $ is compact and $J$ is shape differentiable. Denote the \textit{shape derivative} by
	\ben
	dJ(\Omega ):C^\infty_c(D,\R^d)\rightarrow \R,\quad \theta \mapsto dJ(\Omega)(\theta).
	\een 
	Assuming $dJ(\Omega )$ is of order $k\ge 0$ and $\Gamma$ of class $C^{k+1}$, then there exists a linear and continuous functional $g: C^k(\Gamma)\rightarrow \R$ such that
	\ben\label{volume}
	dJ(\Omega)(\VV)=g(\VV_{|\Gamma}\cdot n),
	\een
\end{theorem}
\bproof
See \cite[pp. 480-481]{MR2731611}.
\eproof


\section{Shape derivatives via averaged adjoint method}\label{section2}
Lagrangian methods in shape optimization allow to compute the shape derivative of functions depending on the solution of partial differential equations without the need to compute the material derivative of the partial differential equations; see \cite{MR2731611} for a description of such a method in the linear case. 
Here we extend  and simplify the averaged adjoint method, a Lagrangian-type method introduced in \cite{sturm13}. 
With this approach the computation of the domain representation of the shape derivative is fast, the retrieval of the boundary form is convenient, 
and no saddle point assumptions is required unlike in \cite{MR2731611}.

Let two vector spaces $E = E(\Omega), F=F(\Omega)$ and  $\tau>0$ be given, and consider a parameterization 
	$\Omega_t = \Phi_t(\Omega)$ for $t\in [0,\tau]$. 
	Ultimately, our goal is to differentiate shape functions of the type 
	$J(\Omega_t)$ which can be written using a Lagrangian as $J(\Omega_t) = \mathcal{L}(\Omega_t, u^t,\hat\psi)$,
	where $u^t\in E(\Omega_t)$ and $\hat\psi\in F(\Omega_t) $.
	The main appeal of the Lagrangian is that 
	we actually only need to compute the derivative with respect to $t$ of $\mathcal{L}(\Omega_t,\hat\varphi,\hat\psi)$ to compute the derivative of $J(\Omega_t)$,  indeed this is the main result of Theorem \ref{thm:sturm}, but this requires a few explanations.

	Since $\mathcal{L}(\Omega_t,\hat\varphi,\hat\psi)$,  is often constituted of integrals on $\Phi_t(\Omega)$, using a change of variable we can rewrite these integrals to integrals on the fixed domain $\Omega$, and consequently transfer the dependence on $t$ to the integrand. 
	However, in the process appear the composed functions $\hat\varphi\circ\Phi_t\in E(\Omega)$ and $\hat\psi\circ\Phi_t\in F(\Omega)$, whose derivatives are not straightforward to compute since $\hat\varphi$ and $\hat\psi$ are defined on the moving spaces $E(\Omega_t)$ and $F(\Omega_t)$.

	Fortunately, and this is the crucial point of the shape-Lagrangian approach,  to compute the shape derivative we can reparameterize the problem by considering
	$\mathcal{L}(\Omega_t, \Psi_t\circ \varphi, \Psi_t\circ \psi)$ instead of $\mathcal{L}(\Omega_t,\hat\varphi,\hat\psi)$, where 
	$\Psi_t$ is an appropriate bijection between
	$E(\Omega)$ and $E(\Omega_t)$, and $\varphi\in E(\Omega)$, $\psi\in F(\Omega)$.
	Now the change of variable in the integrals yields functions $\varphi$ and $\psi$ in the integrand, which are defined on fixed spaces.
	In this paper $E$ and $F$ are $H^1$-spaces,  and in this case we may consider the particular reparameterization
	$\mathcal{L}(\Omega_t,\varphi\circ\Phi_t^{-1},\psi\circ\Phi_t^{-1})$. 
	For spaces such as 
	$H(\curl;\Omega)$, other transformations $\Psi_t$ can be used;  see \cite{HLY,MR2971616,MR2002150}. 

Thus we are led to consider general functions of the type $G:[0,\tau]\times E\times F \rightarrow \R$ with
	$$ G(t,\varphi,\psi):
	=\mathcal{L}(\Phi_t(\Omega),\varphi\circ\Phi_t^{-1},\psi\circ\Phi_t^{-1}).$$
	This is precisely what we do in \eqref{eq:G_on_the_moved_domain} when showcasing an application of the method.
	The main result of this section, Theorem \ref{thm:sturm}, shows that to obtain the shape derivative of $\mathcal{L}$, it is enough to compute the derivative with respect to $t$ of $G$ while assigning the proper values to $\varphi$ and $\psi$. The main ingredient is the introduction of the averaged adjoint equation.

	In addition, in this paper we consider the following specific form
	\ben
	\label{G_lag}
	G(t,\varphi,\psi):= a(t,\varphi,\psi) + b(t,\varphi),
	\een
	where
	$$ a:[0,\tau]\times E \times  F \rightarrow \R, \qquad b:[0,\tau]\times E \rightarrow \R ,$$ 
	are functions such that $\psi \mapsto a(t,\varphi,\psi)$ is linear for all 
	$t\in [0,\tau]$ and $\varphi\in E$. The function $G$ is commonly called {\it Lagrangian}, hence the name of the method.  
	In the applications we have in mind, the function 
	$b$ arises from the objective function while $a$ corresponds to the constraint, after transporting back  to the fixed domain $\Omega$.

Throughout the paper, the Greek letters $\varphi$ and $\psi$ are used for variables, while the roman letters $u,p$ are used for the solutions of the state and adjoint states, respectively.

Let us  assume that for each $t\in [0,\tau]$ the equation
	\ben\label{eq:state_G}
	d_\psi G(t,u^t,0;\hat\psi)= a(t,u^t, \hat\psi) = 0\;\text{ for all } \hat\psi \in  F.
	\een
	admits a unique solution $u^t\in E$. 
Further we make the following assumptions for  $G$.
\begin{thmx}
	\label{amp:gateaux_diffbar_G}
	\label{amp:affine-linear}
	For every $(t,\psi)\in [0,\tau]\times F$
	\begin{enumerate}
		\item[(i)]  $[0,1]\ni s\mapsto G(t,su^t+s(u^t-u^0),\psi)$ is absolutely continuous. 
		\item[(ii)] $[0,1]\ni s\mapsto d_\varphi G(t,su^t+(1-s)u^0,\psi;\hat{\varphi})$ belongs to $L^1(0,1)$ for all $\hat{\varphi}\in E$.
	\end{enumerate}
\end{thmx}
When Assumption \ref{amp:affine-linear} is satisfied,  for $t\in [0,\tau]$ we introduce the \textit{averaged adjoint equation} associated with $u^t$ and $u^0$: 
Find $p^t\in F$ such that 
\begin{equation}\label{averated_}
\int_0^1 d_\varphi G(t,su^t+(1-s)u^0,p^t;\hat{\varphi})\, ds =0 \quad \text{ for all } \hat{\varphi}\in E.
\end{equation}
Notice that, in view Assumption \ref{amp:affine-linear}, for all $t\in[0,\tau]$,
\ben\label{eq:main_averaged}
G(t,u^t,p^t)-G(t,u^0,p^t) = \int_0^1 d_\varphi G(t,su^t+(1-s)u^0,p^t;u^t-u^0)\, ds =0.
\een
We can now state the main result of this section.

	\begin{thmy}
		We assume that
		$$ \lim_{t\searrow 0} \frac{G(t,u^0,p^t)-G(0,u^0,p^t)}{t}=\partial_tG(0,u^0,p^0).$$
	\end{thmy}

	\begin{theorem}
		\label{thm:sturm}
		Let  (H0) and (H1) be satisfied and assume there exists a unique solution $p^t$  of the averaged adjoint equation \eqref{averated_}.
		Then for $\psi \in F$  we obtain
		\begin{equation}\label{eq:dt_G_single}
		\dt b(t,u^t) |_{t=0} = \dt(G(t,u^t,\psi))|_{t=0}=\partial_t G(0,u^0,p^0).
		\end{equation}
	\end{theorem}

\begin{proof}
	Put $g(t) := G(t,u^t,0)-G(0,u^0,0)$, and note that $g(t) = G(t,u^t,\psi)-G(0,u^0,\psi)$ for all $\psi \in F$ and $g(0)=0$. We have to show that 
	$$g'(0):= \lim_{t\searrow 0}\frac{G(t,u^t,0)-G(0,u^0,0)}{t} \quad \text{ exists. }$$  
	Thanks to Assumption \ref{amp:affine-linear} we can define the averaged adjoint $p^t$ and using that $G$ is affine with respect 
	to the third argument, we obtain 
	$$g(t) =  \underbrace{G(t,u^t,p^t)-G(t,u^0,p^t)}_{=0 \mbox{ in view of } \eqref{eq:main_averaged}} + G(t,u^0,p^t)-G(0,u^0,p^t) .$$
	Dividing by $t>0$ and using Assumption (H1) yields
	$$ g'(0) = \lim_{t\searrow 0} \frac{g(t) - g(0)}{t} =  \lim_{t\searrow 0}\frac{G(t,u^0,p^t)-G(0,u^0,p^t)}{t}  = \partial_t G(0,u^0,p^0) $$
	which concludes the proof.
\end{proof}

\begin{remark}
	In terms of $a$ and $b$, equation \eqref{averated_} reads:
	$$ \int_0^1 d_\varphi a(t,su^t+(1-s)u^0,p^t;\hat{\varphi})\, ds = - \int_0^1 d_\varphi b(t,su^t+(1-s)u^0;\hat{\varphi})\, ds $$
for all $\hat{\varphi}\in E$.
	If $\varphi \mapsto a(t,\varphi,\psi)$ is in addition linear, then  \eqref{averated_}  becomes 
	$$        a(t,\hat{\varphi},p^t) = - \int_0^1 d_\varphi b(t,su^t+(1-s)u^0;\hat{\varphi})\, ds $$
 for all $\hat{\varphi}\in E$.
\end{remark}

\section{Tensor representation of the shape derivative}\label{sec:tensor_representation}
In this section we identify tensor representations of the shape derivative that correspond to a large class of problems studied in the literature for PDE-constrained shape optimization. This tensor representation has several interesting properties that we investigate. In particular we exhibit the link between this tensor representation and the usual boundary expression of the shape derivative.
\subsection{Definition and properties}\label{def_prop}
\begin{definition}\label{def:tensor}
Let $\Omega \in \mathfrak P$ be a set with $C^k$-boundary, $k\ge 1$. 
	A shape differentiable function $J$ of order $k$ is said to admit a tensor representation if
	there exist tensors $\Sb_l\in  L^1(D, \mathcal L^l(\R^d,\R^d))$  and $\Sf_l\in L^1(\partial \Omega; \mathcal L^l(\R^d,\R^d))$, $l=0,..,k$,
	such that
	\ben\label{ea:volume_from}
	dJ(\Omega)(\VV) = \sum_{l=0}^k \int_D \Sb_l\cdot D^l\VV \,dx + \int_{\partial \Omega} \Sf_l\cdot D^l_\Gamma \VV\, ds \quad \text{ for all } \VV\in C^k_c(D,\R^d),
	\een
	where $D_\Gamma \theta: = D\theta -(D\theta n)\otimes n$ is the tangential derivative of $\theta$ along $\partial\Omega$. Here 
		$\mathcal L^l(\R^d, \R^d) $ denotes the space of multilinear maps from $\R^d \times \cdots  \times \R^d$ to $\R^d$.
\end{definition}
Most if not all examples involving PDEs from classical textbooks \cite{MR2731611,MR2512810,SokZol92} can be written in the form \eqref{ea:volume_from}.

\begin{remark}
	\begin{itemize}
		\item[(a)] A particular case of the tensor representation \eqref{ea:volume_from}  is the Eshelby energy momentum tensor in continuum mechanics introduced in \cite{MR0489190}; see also \cite{MR3013681}. In this case only $\Sb_1$ is not zero.
		\item[(b)] When $J$ is is shape differentiable in $\Omega$ then by definition  $\theta \mapsto dJ(\Omega)(\theta)$ is a distribution, and if  $\partial \Omega$ is compact, the distribution $\theta \mapsto dJ(\Omega)(\theta)$  is of finite order.
		\item[(c)] If $dJ(\Omega)$ is of order $k=1$ and $ |dJ(\Omega)(\VV)|  \le C\|\VV\|_{H^1(D,\R^d)} $ for all $\theta\in C^\infty_c(D,\R^d)$
		then by density of $C^\infty_c(D,\R^d)$ in $H^1_0(D,\R^d)$ the derivative $dJ(\Omega)$ 
		extends to a continuous functional on $H^1_0(D,\R^d)$, that is,
		$$ |\widehat{dJ(\Omega)}(\VV)|  \le c\|\VV\|_{H^1(D,\R^d)}  \quad \text{ for all } \theta\in H^1_0(D,\R^d).  $$ 
		Therefore by the  theorem of Riesz, we obtain a vector field $W$ in $H^1_0(D,\R^d)$ 
		such that
		$$ \forall \theta\in H^1_0(D,\R^d),\quad  \widehat{dJ(\Omega)}(\VV) = \int_D DW \cdot D \theta  + W \cdot \theta \, dx  $$
		and this defines a tensor representation with $\Sb_1= DW$, $\Sb_0=W$, $\Sf_1=0$ and $\Sf_0=0$.
		\item[(d)] The assumption that $\Omega$ be a set of class $C^k$ can be reduced if $\Sf_l\equiv 0$ for all $0\leq k_0\leq l\leq k$.
	\end{itemize}
\end{remark}

The tensor representation \eqref{ea:volume_from} is not unique in the sense that there might be
several ways to choose the tensors $\Sb_l$ and $\Sf_l$. This is expressed by the fact that 
these tensors are correlated. We describe these relations below in the case $k=1$ in Proposition 
\ref{tensor_relations}, which also describes the link between the tensor representation and the usual boundary representation \eqref{volume} of the shape derivative. 
\begin{proposition}\label{tensor_relations}
	Let $\Omega $ be a subset of $D$ with $C^1$-boundary. Suppose
	that the derivative  $dJ(\Omega)$ has the representation
	\ben\label{eq:first_order_tensor}
	dJ(\Omega)(\VV) = \int_D \Sb_1\cdot D\VV +  \Sb_0\cdot \VV\,dx + \int_{\partial \Omega} \Sf_1\cdot D_\Gamma \VV + \Sf_0\cdot \VV\, ds .
	\een
	If $\Sb_l$ is of class $W^{1,1}$ in $\Omega$ and $D\setminus \overline \Omega$ then indicating by $+$ and $-$ the restrictions of the tensors to $\Omega$ and $D\setminus\overline \Omega$, respectively, we get
	\ben \label{eq:equvilibrium_strong}
	\begin{split}
		-\divv(\Sb_1^+) + \Sb_0^+ &= 0 \quad \text{ in } \Omega \\
		-\divv(\Sb_1^-) + \Sb_0^- &= 0 \quad \text{ in } D\setminus\overline \Omega.
	\end{split}
	\een
	Moreover, 
	we can rewrite the tensor representation  as a distribution on the boundary:
	$$dJ(\Omega)(\VV) = \int_{\partial \Omega} [(\Sb_1^+-\Sb_1^-)n] \cdot\theta + \Sf_1\cdot D_\Gamma \VV + \Sf_0\cdot \VV \, ds $$
	where $n$ denotes the outward unit normal vector to $\Omega$.
	
	If the boundary $\partial \Omega$ is $C^2$ and $\Sf_1\in W^{1,1}(\partial \Omega; \mathcal L^1(\R^d,\R^d))$, then we obtain a more regular distribution, the so-called boundary expression of the shape derivative: 
	\ben\label{eq:general_boundary_exp}
	dJ(\Omega)(\theta) = \int_{\partial \Omega} g_1\, \theta\cdot n\, ds,
	\een
	where
	\begin{equation}
	\label{g1} 
	g_1 :=  [(\Sb_1^+-\Sb_1^-)n]\cdot n + \Sf_0\cdot n + \Sf_1\cdot D_\Gamma n - \divv_\Gamma (\Sf_1^T n) + \mathcal{H}(\Sf_1^T n\cdot n). \end{equation}
	and $\mathcal{H}=\divv_\Gamma n$ denotes the mean curvature\footnote{We define the mean curvature as the sum of the principal curvatures $\kappa_i$, that is, $\mathcal H := \sum_{i=1}^d \kappa_i$.} of $\partial\Omega$ while $\divv_\Gamma := \tr(D_\Gamma)$ is the tangential divergence.  
\end{proposition}
\bproof
Applying Theorem \ref{lem:1} we have
$$dJ(\Omega)(\VV) = \int_D \Sb_1\cdot D\VV +  \Sb_0\cdot \VV\,dx + \int_{\partial \Omega} \Sf_1\cdot D_\Gamma \VV + \Sf_0\cdot \VV\, ds =0 \quad \text{ for all }\VV\in C^1_c(\Omega\cup (D\setminus \overline \Omega), \R^d).  $$
An integration by parts shows \eqref{eq:equvilibrium_strong}.

Then, when $\partial\Omega$ is $C^1$, replacing \eqref{eq:equvilibrium_strong} in the expression of the shape derivative and using Green's formula we obtain  
\ben\label{eq:domain}
\begin{split}
	dJ(\Omega)(\VV) =& \int_{\partial \Omega} \Sf_1\cdot D_\Gamma \VV + \Sf_0\cdot \VV\,\,ds  + \int_{\partial \Omega} [(\Sb_1^+-\Sb_1^-)n]\cdot \VV \, ds\\
	&+ \int_\Omega( \underbrace{-\divv(\Sb_1^+) + \Sb_0^+}_{=0})\cdot \VV \, dx + \int_{D\setminus \overline \Omega} (\underbrace{-\divv(\Sb_1^-) + \Sb_0^-}_{=0})\cdot \VV \, dx\\
	\stackrel{\eqref{eq:equvilibrium_strong}}{=} & \int_{\partial \Omega} [(\Sb_1^+-\Sb_1^-)n]\cdot \VV +\Sf_1\cdot D_\Gamma \VV + \Sf_0\cdot \VV\, \, ds\quad \text{ for all } \VV\in C^1_c(D, \R^d).
\end{split}
\een
With a slight abuse of notation we keep the same notation $n$ for the extension of the normal to a neighborhood of $\partial \Omega$. Let $\theta\in C^1(\overline D,\R^d)$ and define
$\VV_\tau  :=\VV - (\VV\cdot n) n$ the tangential part of $\VV$. Then $\VV_\tau \cdot n =0$  on 
$\partial \Omega$ and hence  from the structure theorem we get $dJ(\Omega)(\VV_\tau)=0$ which yields in view of \eqref{eq:domain}:
\ben\label{eq:boundary_exp}
\begin{split}
	dJ(\Omega)(\VV) &=  dJ(\Omega)((\VV\cdot n)n) \\
	&= \int_{\partial \Omega} ((\Sb_1^+-\Sb_1^-)n\cdot n) (\VV\cdot n) 
	+\Sf_1\cdot D_\Gamma (n(\VV\cdot n)) + (\Sf_0\cdot n) (\VV\cdot n) \ds\\
	& =  \int_{\partial \Omega} ((\Sb_1^+-\Sb_1^-)n\cdot n) (\VV\cdot n) + (\Sf_0\cdot n) (\VV\cdot n)\ds\\
	& + \int_{\partial \Omega} \Sf_1\cdot D_\Gamma n (\VV\cdot n) + n\cdot \Sf_1\nabla_\Gamma (\VV\cdot n) \ds,
\end{split}
\een
where we used that for all functions $f\in C^1(\R^d,\R^d)$ and $g\in C^1(\R^d)$ we have
\ben
\begin{split}
	D(gf) & = g Df + f\otimes \nabla g. 
\end{split}
\een
Finally using $\Sf_1\in W^{1,1}(\partial \Omega; \mathcal L^1(\R^d,\R^d))$ we integrate by parts on the boundary $\partial \Omega$ to transform the last term in \eqref{eq:boundary_exp}
$$
\int_{\partial \Omega} n\cdot \Sf_1\nabla_\Gamma (\VV\cdot n) \ds = \int_{\partial \Omega} (- \divv_\Gamma (\Sf_1^T n) + \mathcal{H}(\Sf_1^T n\cdot n) )(\VV\cdot n) ds.
$$
Therefore \eqref{eq:boundary_exp} reads
\ben\label{eq:boundary_exp_2}
\begin{split}
	dJ(\Omega)(\VV) 
	= & \int_{\partial \Omega} ((\Sb_1^+-\Sb_1^-)n\cdot n) (\VV\cdot n) + (\Sf_0\cdot n) (\VV\cdot n)\ds\\
	& + \int_{\partial \Omega} \Sf_1\cdot D_\Gamma n (\VV\cdot n) + (- \divv_\Gamma (\Sf_1^T n) + \mathcal{H}(\Sf_1^T n\cdot n) )(\VV\cdot n)\ds,
\end{split}
\een
which we can rewrite as \eqref{eq:general_boundary_exp}.
\eproof
\begin{remark}
In Proposition \ref{tensor_relations}, if $\Sf_1\equiv 0$, one can still obtain \eqref{eq:equvilibrium_strong} when $\Omega$ is only Lipschitz instead of $C^1$. 
\end{remark}
\begin{corollary}\label{corollary2}
	Let the assumptions of Proposition~\ref{tensor_relations} be satisfied. Suppose that the tensor 
	$\Sf_1:\partial\Omega\rightarrow \mathcal L(\R^d,\R^d)$ has the form $\Sf_1 = \alpha(I -n\otimes n)$, where $\alpha\in C^0(\partial\Omega)$. Then 
	\eqref{eq:general_boundary_exp} simplifies to 
	\begin{equation}
	\label{shape_der_simplified}
	dJ(\Omega)(\VV) = \int_{\partial \Omega} g_1 \, \theta\cdot n\, \ds, 
	\end{equation}
	where $g_1$ is given by
	$$ g_1 :=  [(\Sb_1^+-\Sb_1^-)n]\cdot n + \Sf_0\cdot n + \alpha\mathcal{H}. $$
\end{corollary}
\bproof
First  $\Sf_1^T = \alpha(I -n\otimes n)^T = \Sf_1$ and $ \Sf_1^T n = \alpha(I - n\otimes n)n = \alpha(n- (n\cdot n)n) = 0$, thus the two last terms in \eqref{g1} vanish. Concerning the third term in \eqref{g1} we write
\begin{align*} 
\Sf_1\cdot D_\Gamma n = \alpha (I -n\otimes n)\cdot D_\Gamma n 
& = \alpha(\tr(D_\Gamma n) - (n\otimes n)\cdot D_\Gamma n )=\alpha(\divv_\Gamma n - (D_\Gamma n n)\cdot n)= \alpha\mathcal{H}.
\end{align*}
where we have used $(D_\Gamma n n)\cdot n=0$. 
\eproof
\begin{remark}
	The particular tensor $\Sf_1 = \alpha(I -n\otimes n)$ in Corollary \ref{corollary2} is commonly encountered in shape optimization problems. In fact, \eqref{shape_der_simplified} corresponds to a standard formula that can be found in most textbooks on shape optimization.
\end{remark}
\begin{remark}
	Recall that for given vector fields $\VV, \zeta$, the second order shape derivative is defined 
	by $$d^2J(\Omega)(\VV)(\zeta) := \dt dJ(\Phi_t^\zeta(\Omega))(\theta)|_{t=0}.$$
	Once we have identified a tensor representation \eqref{ea:volume_from} for the shape derivative $dJ(\Omega)(\theta)$ for fixed $\theta$, it is convenient to differentiate it to also obtain a tensor representation for the second-order shape derivative. Further, Proposition \ref{tensor_relations} or Corollary \ref{corollary2} can also be applied to obtain a boundary expression for the second order shape derivative.
\end{remark}

Similar relations as in Proposition \ref{tensor_relations} could be obtained for any tensor representation of order $k$. For instance in the case $k = 2$ we obtain the relations
\ben \label{eq:equvilibrium_strong_second}
\begin{split}
	\mathcal A\mathbf \Sb_2^+-\divv(\Sb_1^+) + \Sb_0^+ &= 0 \quad \text{ in } \Omega, \\
	\mathcal A\mathbf \Sb_2^- -\divv(\Sb_1^-) + \Sb_0^- &= 0 \quad \text{ in } D\setminus\overline \Omega,
\end{split}
\een
where $(\mathcal A\Sb_2)_l = \sum_{i,j=1}^d \partial_{x_ix_j}(\Sb_2)_{ijl}$. 

Using the averaged adjoint approach from Theorem  \ref{thm:sturm} yields  the tensor representation \eqref{ea:volume_from}  of the shape derivative.
	Then Proposition \ref{tensor_relations} can be used to immediately derive the standard boundary expression of the shape gradient from this  tensor representation.

\subsection{Examples of tensor representations}
In this section we present several examples of representations corresponding to Definition \ref{def:tensor} and apply the observations from  Section \ref{def_prop}.
\subsubsection*{First order tensor representation}
A basic example of a first order tensor representation of the shape derivative is for 
$$J(\Omega) = \int_\Omega f \, dx + \int_{\partial \Omega} g \, ds$$ 
with $f,g\in C^2(\R^d)$. Then one easily computes
$$ dJ(\Omega)(\VV) = \int_\Omega \nabla f \cdot \VV + f\divv(\VV)\,dx + \int_{\partial \Omega} \nabla g\cdot \VV + g\divv_\Gamma \VV\, ds . $$
The corresponding tensor representation \eqref{ea:volume_from} is
\begin{align*}
\Sb_1^+:= f  &I,\quad \Sb_1^-:= 0 , \quad \Sb_0^+:= \nabla f, \quad \Sb_0^-:= 0, \quad
\Sf_1:= g  (I-n\otimes n),\quad  \Sf_0 := \nabla g.
\end{align*}
Note that $\Sf_1$ has the form assumed in Corollary \ref{corollary2}. Applying this Corollary, assuming the domain has enough regularity, we obtain in view of \eqref{shape_der_simplified} the classical formula:
$$ dJ(\Omega)(\VV) =\int_{\partial\Omega} g_1 \, \theta\cdot n\, \ds, 
$$
where $g_1$ is given by
$$ g_1 :=  f + \dn g + g\mathcal{H}. $$
Note that in the particular case $f=0$ we have obtained as a byproduct the formula
\begin{equation}\label{tangential_green}
\int_{\partial \Omega} \nabla g\cdot \VV + g\divv_\Gamma \VV\, ds  =\int_{\partial\Omega} (\dn g + g\mathcal{H}) \, \theta\cdot n\, \ds, 
\end{equation}
and when in addition $\dn g=0$ or $g$ is defined only on $\partial\Omega$, \eqref{tangential_green} becomes the classical {\it tangential Green's formula}; see for instance \cite[proposition 5.4.9]{MR2512810}. 
\subsubsection*{Non-homogeneous Dirichlet problem}
The following problem was already considered for instance in \cite{MR2731611}. Here we present a fairly 
easy way 
to compute the shape derivative. Let $\Omega$ be an open and bounded subset of $\R^d$ 
that is contained in an open and bounded set $D$. Consider
\begin{align}\label{eq:state_non_homo}
-\Delta v &= f\mbox{ in }\Omega,\\
v & = g \mbox{ on }\partial\Omega,\label{eq:state_non_homo_}
\end{align}
where $f\in L^2(D)$ and $g\in H^2(D)$. Introducing the 
variable $u:=v-g$, we observe that \eqref{eq:state_non_homo}-\eqref{eq:state_non_homo_}
is equivalent to the homogeneous Dirichlet problem
\begin{align}\label{eq:state_non_homo-2}
-\Delta u &= \Delta g + f\mbox{ in }\Omega,\\\label{eq:state_non_homo-2_}
u & = 0 \mbox{ on }\partial\Omega.
\end{align}
Consider the cost function
\ben\label{eq:cost_dirichlet_non}
J(\Omega)=\int_\Omega |v-u_d|^2\,dx=\int_\Omega |u+g-u_d|^2\,dx.
\een
The weak formulation of \eqref{eq:state_non_homo-2},\eqref{eq:state_non_homo-2_} reads:  
\ben\label{eq:non_dir}
\mbox{Find } u\in H^1_0(\Omega):  \int_\Omega \nabla u\cdot \nabla \psi \, dx = \int_\Omega  -\nabla g\cdot \nabla \psi +  f\psi\, dx \quad \text{ for all } \psi \in H^1_0(\Omega).
\een
Note that the previous weak formulation is already well-defined for arbitrary open and bounded 
set $\Omega$. We do not need to impose any regularity on $\Omega$.
The perturbed problem of the previous equation, which is obtained by considering \eqref{eq:non_dir} on $\Phi_t(\Omega)$ and
performing a change of variables, reads: find $u^t\in H^1_0(\Omega)$ such that 
\ben\label{eq:non_dir_per}
\int_\Omega A(t)\nabla u^t\cdot \nabla \psi \, dx = \int_\Omega - A(t)\nabla g^t\cdot \nabla \psi +  \xi(t)f^t \psi\, dx \quad \text{ for all } \psi \in H^1_0(\Omega),
\een
where $\xi(t) := \det(D\Phi_t)$ and $A(t) :=\xi(t)D\Phi_t^{-1}D\Phi_t^{-T}$. The following continuity result is standard:
\begin{lemma}
	There exists a constant $c>0$ such that $ \|u^t-u^0\|_{H^1_0(\Omega)} \le ct$ for all $ t\in [0,\tau]. $
\end{lemma}
Introduce 
\begin{align*}
a(t,\varphi,\psi) & :=  \int_\Omega A(t)\nabla \varphi\cdot \nabla \psi \, dx + \int_\Omega A(t)\nabla g^t\cdot \nabla \psi  - \xi(t)f^t \psi\, dx  \\
b(t,\varphi) &  :=\int_\Omega \xi(t) |\varphi +g^t-u_d^t|^2\,dx.
\end{align*}
where $g^t:=g\circ \Phi_t$ and $u_d^t:=u_d\circ \Phi_t$.

Recall that the associated Lagrangian \eqref{G_lag} is
$G(t,\varphi,\psi) = a(t,\varphi,\psi) + b(t,\varphi).
$
The averaged adjoint equation \eqref{averated_} reads 
$$\int_\Omega A(t)\nabla \varphi\cdot \nabla p^t \, dx = \int_\Omega (u^t+u^0 + 2g^t - 2u_d^t)\varphi\,dx \quad \text{ for all } \varphi\in H^1_0(\Omega). $$

The following continuity result for the adjoint is standard:
\begin{lemma}
	There exists a constant $c>0$ such that 
	$$ \|p^t-p^0\|_{H^1(\Omega)} \le ct \quad \text{ for all } t\in [0,\tau].$$
\end{lemma}
One readily verifies that all assumptions of Theorem~\ref{thm:sturm} are satisfied, except for (H1) which we now prove.
	Indeed using $p^t\rightarrow p^0$ in $H^1_0(\Omega)$ as $t$ goes  to zero and the strong differentiability of $t\mapsto A(t)$ and $t\mapsto \xi(t)$, we get 
	\begin{align*}
	&\lim_{t\searrow 0 }  \frac{ G(t,u^0,p^t) - G(0,u^0,p^t)}{t}   \\
	&\hspace{1cm} 
	=\lim_{t\searrow 0} \bigg(\int_{\Omega} \left(\frac{ A(t) - I }{t} \right) \nabla u^0 \cdot \nabla p^t \, dx + \int_{\Omega} \left(\frac{ A(t)\nabla g^t - \nabla g }{t}\right) \cdot \nabla p^t {\color{red}\, dx} - 
	\left(\frac{\xi(t)f^t - f}{t} \right) p^t \, dx \\
	& \hspace{1cm} + \int_{\Omega} \frac{ \xi(t)| u^0 + g^t - u_d^t |^2  - | u^0 + g - u_d  |^2}{t}\bigg)\\
	&\hspace{1cm} = \partial_t G(0,u^0,p^0),
	\end{align*}
	which shows that (H1) is satisfied.

Hence, applying Theorem~\ref{thm:sturm} yields
$$ dJ(\Omega)(\theta) = \partial_t a(0,u,p) + \partial_tb(0,u) ,$$
which is by definition equivalent to 
\ben
\begin{split}
	dJ(\Omega)(\theta)    = &\int_\Omega A'(0)(\nabla u \cdot \nabla p + \nabla g\cdot \nabla p) \, dx + \int_\Omega \nabla(\nabla g\cdot \theta)\cdot \nabla p -  \divv(f\theta) p \, dx\\
	& + \int_\Omega \divv(\theta) |u+g-u_d|^2 + (\nabla (g-u_d)\cdot \theta))(u+g-u_d)\,dx.
\end{split}
\een
Since $A'(0) = (\di \VV) I - D\VV^T - D\VV$ we obtain the tensor representation \eqref{eq:first_order_tensor} with:
\begin{align*} 
\Sb_1 &= I(\nabla u\cdot \nabla p  + \nabla g \cdot \nabla p - fp + |u+g-u_d|^2)  -\nabla u\otimes \nabla p - \nabla p \otimes \nabla u - \nabla p \otimes \nabla g, \\
\Sb_0 &= D^2 g\nabla p - p\nabla f  
+(u+g-u_d)  \nabla (g-u_d),   \\
\Sf_1 &= 0, \quad \Sf_0 = 0. 
\end{align*}
Now applying  \eqref{shape_der_simplified} we get immediately
\begin{align*}
g_1 = \nabla u\cdot \nabla p +\nabla  g \cdot \nabla p - fp + |u+g-u_d|^2  -2\partial_n u\partial_n p - \partial_n p \partial_n g.
\end{align*}
Using the definition of the tangential gradient and $p=0,u=0$ on $\Gamma$ implies $\nabla_\Gamma u=\nabla_\Gamma p=0$, so we obtain the simpler expression 
\begin{align*}
g_1 = -\partial_n u\partial_n p  + |u+g-u_d|^2 = -\partial_n (v-g)\partial_n p  + |v-u_d|^2 .
\end{align*}
Finally, substituting back $u=v-g$ we obtain the formula
$$dJ(\Omega)(\VV) = \int_{\partial \Omega} (-\partial_n (v-g)\partial_n p  + |v-u_d|^2)\, \theta\cdot n\ds  .$$
This formula can be found for instance in \cite[p. 566, Formula 6.38]{MR2731611}, where
the adjoint has the sign opposite to our case.
\subsubsection*{Elliptic problem: first order tensor representation}
Suppose that $\Omega \subset D\subset \R^d$ is a smooth bounded domain, where $D\subset \R^d$ is the smooth ``universe''. 
Let us  consider the Dirichlet problem:
\begin{align}\label{eq:state}
\begin{split}
-\divv(M\nabla u) + u&= f\mbox{ in }\Omega,\\
u & = 0 \mbox{ on }\partial\Omega,
\end{split}
\end{align}
where $M\in \R^{d,d}$ is a positive definite matrix.
Consider the cost function
\ben\label{eq:cost_dirichlet}
J(\Omega)=\int_\Omega |u-u_d|^2\,dx,
\een
where $u_d\in C^1(\R^d)$. Let us introduce 
\begin{align*}
a(t,\varphi,\psi) & :=  \int_\Omega (M Q^t\nabla \varphi \cdot Q^t\nabla \psi   + \varphi \psi )\xi(t) \, dx - \int_\Omega \xi(t) f^t\psi \, dx \\
b(t,\varphi) &  :=\int_\Omega \xi(t) |\varphi-u_d^t|^2\,dx,
\end{align*}
where $Q^t:=D\Phi_t^{-T}$ and $\xi(t) := \det(D\Phi_t)$.
Then the weak formulation of \eqref{eq:state} on the perturbed domain $\Omega_t$, once transported back to $\Omega$ is
\benn
a(t,u^t,\psi) = 0 \quad \text{ for all } \psi\in H^1_0(\Omega).
\eenn
The Lagrangian corresponding to the minimization of $J(\Omega)$ and the PDE constraint \eqref{eq:state} is 
\ben\label{eq:lagrange}
G(t,\varphi,\psi) = b(t,\varphi) + a(t,\varphi,\psi). 
\een
 It can be shown using Theorem \ref{thm:sturm}  that $dJ(\Omega)(\VV) =  \partial_t G(0,u,p)$, where $p\in H^1_0(\Omega)$ denotes the adjoint state:
\ben\label{eq:adjoint_state_simple_dirichlet}
\begin{split}
	\int_\Omega M\nabla \psi\cdot \nabla p +p\psi\, dx=-\int_\Omega 2 (u-u_d)\psi\,  dx\quad \text{ for all } \psi \in H^1_0(\Omega) .
\end{split}
\een
\\
The tensor representation \eqref{ea:volume_from} of the shape derivative of $J(\Omega)$ in direction $\VV\in C^2_c(D,\R^d)$ is  given by
\ben\label{eq:shape_derivative}
dJ(\Omega)(\VV) = \int_\Omega \mathbf S_1\cdot  D\VV +  \Sb_0\cdot \VV \, dx.
\een
where we use the relation $(\nabla p \otimes \nabla u)\cdot D\VV = D\VV \nabla u\cdot \nabla p$ to get the tensors 
\begin{align}
\label{f_mathbf}
\Sb_0 & = -2(u-u_d)\nabla u_d - p\nabla f ,\\
\Sb_1 & = - \nabla p \otimes M \nabla u -  \nabla u \otimes M^T\nabla p + (M\nabla u\cdot \nabla p+ up-fp+(u-u_d)^2)I.  
\end{align}
In the simple case where $M=I$, assuming $u,p\in C^2(\overline \Omega)$, we know from the previous discussion that \eqref{eq:equvilibrium_strong} 
 is satisfied. 
Noting that 
\begin{align*}
\divv(\nabla p\otimes \nabla u) &= \Delta u \nabla p + (D^2p)^T\nabla u, \\
\divv(\nabla u\otimes \nabla p) &= \Delta p \nabla u + (D^2u)^T\nabla p, \\ 
\nabla (\nabla u\cdot \nabla p) &= D^2u\nabla p +  D^2 p \nabla u,
\end{align*}
the relation
\ben\label{nes}
- \divv(\Sb_1) + \Sb_0 = 0 \quad \text{ in } \Omega
\een
is equivalent to
$$ (-\Delta u + u -f)\nabla p + (-\Delta p + p + 2(u-u_d))\nabla u = 0  \quad \text{ in } \Omega.$$ 
Therefore, we observe that the fundamental relation \eqref{eq:equvilibrium_strong} between the tensors $\Sb_1$ and $\Sb_0$  corresponds to the  strong solvability of the state and adjoint state equation.
\section{Descent directions}\label{section3}

In this paper we are interested in numerical methods for shape optimization problems of the type
\ben
\min_{\Omega \in \mathfrak{P}}J(\Omega),
\een
where $\mathfrak{P}\subset\mathcal{P}(D)$ is the admissible set. Assume  $J:\mathfrak{P}\rightarrow \R$ is shape differentiable at $\Omega\subset D\subset \R^d$.
\begin{definition}[descent direction]\label{def:descent}
	The vector field $\VV\in C^{0,1}_c(D,\R^d)$ is called a {\it descent direction} for $J$ at $\Omega$ if there exists an $\ve$ such that
	$$ J(\Phi_t^\theta(\Omega))< J(\Omega)\mbox{ for all } t\in (0,\ve).$$
	If the Eulerian semiderivative of $J$  at $\Omega$ in direction $\VV$ exists and if it is a descent direction then by definition
	\ben\label{eq:descent}
	dJ(\Omega)(\VV) < 0.
	\een
\end{definition}
Descent directions are used in iterative methods for finding approximate (possibly local) minimizers of  $J(\Omega)$. Typically, at a given starting point $\Omega$, one determines a descent direction $\VV$ and proceeds along this direction as long as the cost functional $J$ reduces sufficiently using a step size strategy.
In this section we give a general setting for computing descent directions in the framework of gradient methods using the domain and boundary representations of the shape derivative according to Theorem \ref{thm:structure_theorem}. We show how a descent direction $\VV$ with any regularity $H^s$, $s\geq 1$ can be obtained by solving an appropriate partial differential equation. 
We also show how to deal with bound constraints on $\VV$. 
In order to develop a setting allowing to define general descent directions, we recall sufficient conditions for the solvability of the following operator equation
\benn
A\VV=f,
\eenn
where $A:\Eb\rightarrow \Eb'$ is an operator between a Banach space $\Eb$ and its dual $\Eb'$.  Sufficient conditions
for the bijectivity of $A$  are given by the theorem of Minty-Browder \cite[p.364, Theorem 10.49]{MR2028503}.
\begin{theorem}[Minty-Browder]\label{MB}
	Let $(\Eb;\|\cdot\|_{\Eb})$ be a reflexive separable Banach space and $A:\Eb\rightarrow \Eb'$ a bounded, hemi-continuous, monotone and coercive operator. Then $A$ is surjective, i.e. for each $f\in \Eb'$ there exists $\VV\in \Eb$ such that $A\VV=f$.
	Moreover if $A$ is strictly monotone then it is bijective.
\end{theorem}

Let $A:\Eb\rightarrow \Eb'$ be an operator on a reflexive, separable Banach space $\Eb$ satisfying the assumptions of  Theorem \ref{MB} with $A(0)\VV\ge 0$ for $v\in \Eb$. 
Assume $dJ(\Omega)$ can be extended to $\Eb'$ if necessary; for simplicity we keep the same notation for the extension. Introduce the bilinear form

	\ben\label{eq:bilinear}
	\bil:\Eb\times \Eb\rightarrow \R,\qquad \bil(\VV,\zeta):=\langle A\VV,\zeta \rangle_{\Eb',\Eb}.
	\een

Consider the variational problem:
\ben
\label{VP_1}
(\text{VP})\qquad
\mbox{Find  } \VV_1\in \Eb\mbox{  such that  }  \bil(\VV_1,\zeta)=- dJ(\Omega)(\zeta) \mbox{  for all  } \zeta\in \Eb ,
\een
Then the solution $\VV_1$ of (VP) is a descent direction since $dJ(\Omega)(\VV_1) = -\bil(\VV_1,\VV_1)\leq 0$.

In certain situations it is desirable to have bound constraints on the shape perturbation. This may be handled by considering the more general case of a variational inequality. Given a subset $K\subset \Eb$ with $0\in K$, consider the variational inequality:
\benn
(\text{VI})\qquad
\mbox{Find  }\VV_2\in K\mbox{  such that  } \bil(\VV_2,\VV_2-\zeta) \le  dJ(\Omega)(\zeta-\VV_2) \mbox{  for all  } \zeta\in K .
\eenn
The solution $\VV_2$ of  (VI) yields a descent direction for $J$ at $\Omega$ since taking $\zeta = 0\in K$ we get 
$$dJ(\Omega)(\VV_2) \le -\bil(\VV_2,\VV_2)\leq 0.$$ In view of Theorem \ref{thm:structure_theorem}, we choose $\Eb\subset H^{s}(D)$ where
$s$ is such that $dJ(\Omega):H^s(D,\R^d)\rightarrow \R^d$ is continuous. When $\Eb$ is a Hilbert space, one may identify $\Eb'$ with $\Eb$. Therefore if $\bil$ is bilinear, coercive, and continuous, then Lax Milgram's lemma ensures that (VP) has a unique solution. For all other cases we may have to use Theorem \ref{MB} or similar results.

\begin{remark}
		\begin{itemize}
			\item[(a)]
			Let $\Eb:= H^1_0(D,\R^d)$, 
			$\bil(\VV, \zeta) := \int_D D\VV:D\zeta \,dx $ and $\Omega^+\Subset D$. Then 
			\eqref{VP_1} reads: find $\theta \in H^1_0(D,\R^d)$, such that
			$ \mathcal B(\theta, \zeta) = - dJ(\Omega^+)(\zeta)$ for all $\zeta\in H^1_0(D,\R^d)$. Under the assumption that $\partial \Omega^+ \in C^2$, $\theta|_{\Omega^+} \in H^2(\Omega^+)$,
			and $\theta|_{D\setminus \overline{\Omega^+}} \in H^2(D\setminus \overline{\Omega^+})$, Proposition \ref{tensor_relations} yields
			$$ \int_{\partial \Omega^+} g \; \zeta\cdot n \, ds = dJ(\Omega^+)(\zeta) \quad \text{ for all }\zeta\in H^1_0(D,\R^d), \quad \text{ where } \quad g= -(D\theta^+ n -D\theta^- n)\cdot n. $$ 
			This shows that the  restriction to $\partial \Omega^+$  of the obtained descent direction $\theta$ is 
			more regular than the function $g$.	
			\item[(b)]  Let $n$ be an extension of the unit normal to $\Omega^+$ in  $D$. If $\theta$ defined on $D$ is a descent direction then $(\theta\cdot n)n|_{\partial\Omega^+}$ is also a descent direction, for the tangential part of $\theta$ does not influence the derivative. Indeed define  $\theta_{\tau}:= \theta - (\theta\cdot n)n$, then by Nagumo's theorem $dJ(\Omega^+)(\theta_{\tau}) = 0$ and thus $dJ(\Omega^+)(\theta) = dJ(\Omega^+)((\theta\cdot n)n)$.
			However, $\theta$ and $(\theta\cdot n)n$ lead  to different transformations of the domains in general, indeed the tangential term actually has an influence for large deformations,  which means $\Phi_t^\theta(\Omega^+)  \ne \Phi_t^{(\theta \cdot n)n}(\Omega^+)$. This influence appears for instance in the shape Hessian.
		\end{itemize}
	\end{remark}

 	\section{Electrical impedance tomography}\label{section4}
	We consider an application of the results above to a typical and important interface problem: the inverse problem of electrical impedance tomography (EIT) also known as the inverse conductivity or Calder\'on's problem \cite{MR590275} in the mathematical literature. It is an active field of research with an extensive literature; for further details we point the reader toward the survey papers \cite{MR1955896,MR1669729} as well as \cite{MR2986262} and the references therein. We consider the particular case where the objective is to reconstruct a piecewise constant  conductivity $\sigma$ which amounts to determine an interface ${\Gamma^+}$ between some inclusions and the background. We refer the reader to \cite{MR2329288, MR2132313, MR2536481,HintLauNov,MR2886743, 0266-5611-31-7-075009,Canelas2013} for more details on this approach. 
	
	The main interest of studying EIT is to apply the approach developed in this paper to a problem which epitomizes general interface problems and simultaneously covers the entire spectrum of difficulties encountered with severely ill-posed inverse problem.
	\subsection{Problem statement}\label{sec:eit}
	Let $D\subset \R^d$ be a Lipschitz domain, and $\Omega^+,\Omega^-\subset D$ open sets such that $D=\Omega^+\cup \Omega^-\cup \Gamma^+$, where ${\Gamma^+}=\partial \Omega^+=\overline{\Omega^+}\cap \overline{\Omega^-}$ and $\Gamma = \partial D=\partial\Omega^-\setminus\Gamma^+$; see Figure \ref{partition}. In this section $n$ denotes either the outward unit normal vector to $D$ or the outward unit normal vector to $\Omega^+$. Decompose $\Gamma$ as $\Gamma = \Gamma_d\cup\Gamma_n$.  Let $\sigma=\sigma^+\chi_{\Omega^+}+\sigma^-\chi_{\Omega^-}$ where $\sigma^\pm$ are scalars and $f=f^+\chi_{\Omega^+}+f^-\chi_{\Omega^-}$ where $f^\pm\in H^1(D)$.
	\begin{figure}
		\begin{center}
			\includegraphics[width=0.3\textwidth]{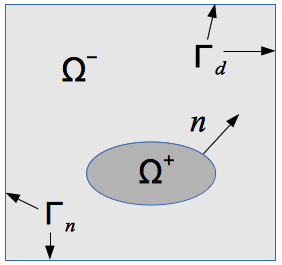}
			\caption{Partition $D=\Omega^+\cup \Omega^-\cup \Gamma$. }\label{partition}
		\end{center}
	\end{figure}
	Consider the following problems: find $\une\in H^1_d (D)$ such that
	\begin{equation}\label{varun}
	\int_D \sigma \nabla \une \cdot \nabla z = \int_D fz + \int_{\Gamma_n} gz\ \mbox{ for all }\ z\in H^1_d (D) 
	\end{equation}
	and find $\udi\in H^1_{dn}(D)$ such that 
	\begin{equation}\label{varud}
	\int_D \sigma \nabla \udi \cdot \nabla z = \int_D fz\  \mbox{ for all }\ z\in H^1_0(D) 
	\end{equation}
	where
	\begin{align*}
	H^1_d (D)&:=\{v\in H^1(D)\ |\ \; v=0\mbox{ on }\Gamma_d\}, \\
	H^1_{dn}(D)&:=\{v\in H^1(D)\ |\ \; v=0\mbox{ on }\Gamma_d, v=h\mbox{ on }\Gamma_n\},\\ 
	H^1_{0}(D)&:=\{v\in H^1(D)\ |\ \; v=0\mbox{ on }\Gamma\}
	\end{align*}
	and $g\in H^{-1/2}(\Gamma_n)$ represents the input, in this case the electric current applied on the boundary and $h\in H^{1/2}(\Gamma_n)$ is the measurement of the potential on $\Gamma_n$, or the other way around, i.e. $h$ can be the input and $g$ the measurement.  Define also the space
	$$PH^k(D) := \{ u = u^+\chi_{\Omega^+} + u^-\chi_{\Omega^-}|\ u^+\in H^k(\Omega^+),\ u^-\in H^k(\Omega^-) \}. $$
	Consider the following assumption which will be used only to derive the boundary expression of the shape derivative but is not required for the domain expression:
	\begin{assumption}\label{assump1}
		The domains $D, \Omega^+,\Omega^-$ are of class $C^k$, $f\in PH^{\max(k-2,1)}(D)$, $g\in H^{k-\frac{3}{2}}(D)$ and  $h\in H^{k-\frac{1}{2}}(D)$ for $k\geq 2$. 
	\end{assumption}
	Applying Green's formula under Assumption \ref{assump1},  equations \eqref{varun} and \eqref{varud}  are equivalent to the following transmission problems where $\une = \une^+\chi_{\Omega^+} + \une^-\chi_{\Omega^-}$ and $\udi = \udi^+\chi_{\Omega^+} + \udi^-\chi_{\Omega^-}$:
	\begin{align}
	\label{eit1.1_}-\sigma^+\Delta \une^+ & = f \mbox{ in }\Omega^+, \quad  -\sigma^-\Delta \une^- = f \mbox{ in }\Omega^-, \\
	\label{eit1.2_}  \une^- & = 0 \mbox{ on }\Gamma_d ,\\
	\label{eit1.3_}\sigma^- \dn \une^- & = g \mbox{ on }\Gamma_n ,
	\end{align}
	\begin{align}
	\label{eit2.1_}-\sigma^+\Delta \udi^+ & = f \mbox{ in }\Omega^+,\quad  -\sigma^-\Delta \udi^- = f \mbox{ in }\Omega^-,\\
	\label{eit2.2_}\udi^- & = 0 \mbox{ on }\Gamma_d, \\
	\label{eit2.3_}\udi^- & = h \mbox{ on }\Gamma_n, 
	\end{align}
	with the transmission conditions
	\ben
	\begin{split}
		\sigma^+\partial_n\une^+=\sigma^-\partial_n\une^-, &  \qquad \sigma^+\partial_n\udi^+=\sigma^-\partial_n\udi^- \quad \text{ on } {\Gamma^+},\\ 
		\une^+=\une^-,     &\qquad  \udi^+=\udi^-\quad \text{ on } {\Gamma^+}. \end{split}
	\een
	On $\Gamma_d$ we impose homogeneous Dirichlet conditions, meaning that the voltage is fixed and no measurement is performed. One may take $\Gamma_d =\emptyset$, in which case   \eqref{varun} becomes a pure Neumann problem and  additional care must be taken for the uniqueness and existence of a solution. The situation $\Gamma_d \neq \emptyset$ corresponds to partial measurements. Alternatively, it is also possible to consider a slightly different problem where each function $\une$ and $\udi$ has both the boundary conditions \eqref{eit1.3_} and \eqref{eit2.3_} on different parts of the boundary. 
	
	Several measurements can be made by choosing sets of functions $\{g_i\}_{i=1}^I$ and $\{h_i\}_{i=1}^I$. 
	Writing $u_{n,i}$ and $u_{d,i}$ for the corresponding states, the problem of electrical impedance tomography is: 
		\begin{equation}\label{EIT_pb}
		\mbox{(EIT): Given $\{g_i\}_{i=1}^I$ and $\{h_i\}_{i=1}^I$,  find $\sigma$ such that $u_{n,i} = u_{d,i}$ in $D$ for $i=1,..,I$.}
		\end{equation}
		Note that $u_{n,i} = u_{n,i}(\Omega^+)$ and $u_{d,i} = u_{d,i}(\Omega^+)$ actually depend on $\Omega^+$ through $\sigma=\sigma(\Omega^+)$, however we  often write $u_{n,i}$ and $u_{d,i}$ for simplicity. In this section, we assume that the conductivities $(\sigma^+,\sigma^-)$ are known, therefore the EIT problem \eqref{EIT_pb} reduces to the following shape optimization problem where $\Omega^+$ is the unknown. 
		\begin{align}\label{EIT-SO}
		\begin{split}
		\mbox{Given }& \mbox{$\{g_i\}_{i=1}^I$, $\{h_i\}_{i=1}^I$ and $(\sigma^+,\sigma^-)$ with $\sigma=\sigma^+\chi_{\Omega^+}+\sigma^-\chi_{\Omega^-}$,} \\
		& \mbox{find $\Omega^+$ such that $u_{n,i} = u_{d,i}$ in $D$ for $i=1,..,I$.}
		\end{split}
		\end{align}
		Actually, the result for several measurements can be straightforwardly deduced from the case of one measurement by summing  the cost functionals corresponding to each measurement, therefore in this section we stick to the  case $I=1$ of one measurement $g$ for simplicity of presentation. In section \ref{section6} we consider several measurements for the numerics.

	The notion of well-posedness due to Hadamard requires the existence and uniqueness of a solution and the continuity of the inverse mapping. The severe ill-posedness of EIT is well-known: uniqueness and continuity of the inverse mapping depend on the regularity of $\sigma$, the latter being responsible for the instability of the reconstruction process. Additionally, partial measurements often encountered in practice render the inverse problem even more ill-posed. We refer to the reviews \cite{MR1955896,MR1669729, MR2986262} and the references therein for more details. A standard cure against the ill-posedness is to regularize the inverse mapping. In this paper the regularization is achieved by considering smooth perturbations of the domains $\Omega^+$.
	
	To solve problem \eqref{EIT-SO}, we use an optimization approach by minimizing the shape functionals
	\begin{align}
	\label{eit3.1}J_1(\Omega^+) &= \frac{1}{2}\int_D (\udi(\Omega^+) - \une(\Omega^+))^2 ,\\
	\label{eit3.2}J_2(\Omega^+) &= \frac{1}{2}\int_{\Gamma_n} (\une(\Omega^+) - h)^2. 
	\end{align}
	Since $\udi,\une\in H^1(D)$ and $h\in H^{1/2}(\Gamma_n)$, $J_1$ and $J_2$ are well-defined. Note that $J_1$ and $J_2$ are redundant for the purpose of the reconstruction but our aim is to provide an efficient way of computing the shape derivative of two functions which are often encountered in the literature.  To compute these derivatives we use the approach described in Section \ref{section2}. First of all introduce
	\begin{align}
	\label{eit4.1}F_1(\phdi,\phne) & := \frac{1}{2}\int_D (\phdi - \phne)^2, \\
	\label{eit4.2}F_2(\phne) & := \frac{1}{2}\int_{\Gamma_n} (\phne - h)^2. 
	\end{align}
	Note that $J_1(\Omega^+) = F_1(\udi(\Omega^+),\une(\Omega^+))$ and $J_2(\Omega^+) =F_2(\une(\Omega^+))$. 
	Next consider $\mathfrak{P}$ a subset of $\mathcal{P}(D)$ and the Lagrangian $\mathcal{L}:\mathfrak{P}\times H^1_d(D)\times H^1_d (D)\times H^1_0 (D)\times H^1_d (D)\rightarrow \R$: 

\begin{align}
\label{G}
\begin{split}		
\mathcal{L}(\Omega^+,\varbf,\psib) & := \alpha_1 F_1(\phdi,\phne) + \alpha_2 F_2(\phne)\\ 
& + \int_D \sigma \nabla \phdi\cdot\nabla \psdi - f\psdi  
+ \int_{\Gamma_n} -\sigma^-\partial_n\psdi (\phdi - h)\\ 
&+ \int_D \sigma \nabla \phne\cdot\nabla \psne - f\psne - \int_{\Gamma_n} g\psne,
\end{split}		
\end{align}
	where $\varbf :=(\phdi,\phne)$ and $\psib :=(\psdi,\psne)$. 
	The term $-\sigma^-\partial_n\psdi$ in the second integral of \eqref{G} is used to enforce the boundary condition \eqref{eit2.3_}.
	Introduce the objective functional
	\begin{align*}
	J(\Omega^+) := \alpha_1 J_1(\Omega^+) + \alpha_2 J_2(\Omega^+).
	\end{align*}
	%

To compute the derivative of the Lagrangian depending on \eqref{eit1.1_}-\eqref{eit2.3_} we apply the averaged adjoint method from Section \ref{section2}. 
	

	\subsection{State and adjoint equations}
	The state $\ub:=(\udi,\une)$ and adjoint state $\pb:=(\pdi,\pne)$ are solutions of the equations: 
	\begin{align}
		\label{lag1} \partial_{\psib}  \mathcal{L}(\Omega^+,\ub, \pb)(\hat \psib) &= 0\mbox{ for all } \hat \psib \in H^1_0(D)\times H^1_d(D),\\
		\label{lag2} \partial_{\varbf}   \mathcal{L}(\Omega^+, \ub ,\pb)(\hat \varbf) &= 0\mbox{ for all } \hat \varbf \in H^1_d(D)\times H^1_d(D) .
		\end{align}
	
	Writing \eqref{lag1} explicitely, one can obtain easily the state equations \eqref{varun} and \eqref{varud}. 
	Then \eqref{lag2} yields the equation for the adjoint $\pdi$:
	$$\partial_{\phdi} \mathcal{L}(\Omega^+,\ub,\pb)(\hat{\varphi}_d) = 0,\mbox{ for all } \hat{\varphi}_d \in H^1_d (D),$$
	which leads to 
	\begin{align}\label{eq:lambda}
	&\int_D \sigma \nabla  \pdi\cdot \nabla \hat{\varphi}_d \dx  = - \alpha_1 \int_D (\udi - \une) \hat{\varphi}_d \dx -\int_{\Gamma_n} -\sigma^-\partial_n\pdi\hat{\varphi}_d \ds \quad \text{ for all } \hat{\varphi}_d \in H^1_d (D)
	\end{align}
	which is the variational formulation for the adjoint state $\pdi$. This yields the following variational formulation when test functions are restricted to $H^1_0(D)$:
	\begin{align}\label{eit4b.0}
	& \int_D \sigma \nabla  \pdi\cdot \nabla \widetilde\varphi \dx = -\alpha_1 \int_D (\udi - \une) \widetilde\varphi\dx\quad  \text{ for all } \widetilde\varphi\in  H^1_0(D).
	\end{align}
	If we use Assumption~\ref{assump1}, we get $\pdi\in PH^k(D)$ and using Green's formula in $\Omega^+$ and $\Omega^-$ with  $\widetilde\varphi\in C_c^\infty(\Omega^+)$ 
	and $\widetilde\varphi \in C_c^\infty(\Omega^-)$, we obtain the strong form
	\begin{align}
	\label{eit4b.1}-\divv(\sigma \nabla \pdi) & = -\alpha_1(\udi - \une)\mbox{ in }\Omega^+ \mbox{ and } \Omega^-.
	\end{align}
	Hence using now Green's formula in \eqref{eq:lambda} and using 
	\eqref{eit4b.1} gives 
	\begin{align*}
	& \int_{\Gamma^+} [\sigma \dn \pdi]_{\Gamma^+} \hat{\varphi}_d \ds  + \int_{\Gamma_n} (\sigma\dn \pdi  - \sigma^-\partial_n\pdi) \hat{\varphi}_d\ds =0 \quad \text{ for all } \hat{\varphi}_d\in H^1_d (D),
	\end{align*}
	where $[\sigma \dn \pdi]_{\Gamma^+} = \sigma^+ \dn \pdi^+ - \sigma^- \dn \pdi^-$ is the jump of $\sigma \dn \pdi$ across $\Gamma^+$. Since 
	the integral on $\Gamma_n$ above vanishes and $p_d\in H^1_0(D)$, we obtain 
	\begin{align}
	\label{eit4b.3} \pdi & = 0 \mbox{ on }\Gamma,\\
	\label{eit4b.4} \sigma^+ \dn \pdi^+ & = \sigma^- \dn \pdi^-  \mbox{ on }\Gamma^+.
	\end{align}
	
	Finally solving 
	$$\partial_{\phne} \mathcal{L}(\Omega^+,\ub,\pb)(\hat{\varphi}_n) = 0,  \mbox{ for all } \hat{\varphi}_n\in H^1_d (D),$$
	leads to the variational formulation
	\begin{align}\label{eit5}
	& \int_D -\alpha_1(\udi - \une)\hat{\varphi}_n +\sigma \nabla \pne\cdot \nabla \hat{\varphi}_n  + \int_{\Gamma_n}  \alpha_2(\une - h )\hat{\varphi}_n =0
	\end{align} 
	for all $\hat{\varphi}_n \in H^1_d (D)$.

	Similarly as for $\pdi$ we get, under Assumption~\ref{assump1}, $\pne\in PH^k(D)$ and the strong form 
	\begin{align}
	\label{eit4.3} -\divv(\sigma \nabla \pne) & = \alpha_1(\udi - \une)\mbox{ in }\Omega^+ \mbox{ and } \Omega^-,\\
	\label{eit4.4} \sigma\dn \pne & = -\alpha_2(\une - h )\mbox{ on }\Gamma_n,\\
	\label{eit4.5} \pne & = 0\mbox{ on }\Gamma_d,\\
	\sigma^+\partial_n\pne^+ & =\sigma^-\partial_n \pne^- \text{ on } {\Gamma^+}, \quad  \pne^+=\pne^- \text{ on } {\Gamma^+}.
	\end{align}
 	\subsection{Shape derivatives}
	
	Let us consider a transformation $\Phi_t^{\VV}$ defined by \eqref{Vxt} with $\VV\in C^1_c(D,\R^d)$. Note that $\Phi_t^{\VV}(D) = D$ but in general $\Phi_t^{\VV}(\Omega^+) \neq \Omega^+$. We use the notation $\Omega^+(t):= \Phi_t^{\VV}(\Omega^+)$. Our aim is to show the shape differentiability of $J(\Omega^+)$ with the
	help of Theorem~\ref{thm:sturm}. 
	Following the methodology described in Section \ref{section2} we introduce
	\ben\label{eq:G_on_the_moved_domain}
	G(t,\varbf,\psib):=\mathcal{L}(\Omega^+(t),\varbf\circ\Phi_t^{-1},\psib\circ\Phi_t^{-1}).
	\een
We proceed to the  change of variables $\Phi_t(x)=y$ in \eqref{eq:G_on_the_moved_domain}  to get the canonical form \eqref{G_lag}. First of all let us denote $f_{\Omega^+(t)}=f^+\chi_{\Omega^+(t)}+f^-\chi_{\Omega^-(t)}$ and $\sigma_{\Omega^+(t)}=\sigma^+\chi_{\Omega^+(t)}+\sigma^-\chi_{\Omega^-(t)}$; recall that $\sigma^\pm$ are scalars but $f^{\pm}$ are functions. Then note that   the  change of variables $\Phi_t(x)=y$ leads to considering the following functions inside the integrals:
\begin{align*}
\sigma_{\Omega^+(t)}\circ \Phi_t
& =\sigma^+\chi_{\Omega^+(t)}\circ \Phi_t+\sigma^-\chi_{\Omega^-(t)}\circ \Phi_t = \sigma^+\chi_{\Omega^+}+\sigma^-\chi_{\Omega^-}= \sigma,\\
f_{\Omega^+(t)}\circ \Phi_t
& =f^+\circ \Phi_t\, \chi_{\Omega^+(t)}\circ \Phi_t
+f^-\circ \Phi_t\, \chi_{\Omega^-(t)}\circ \Phi_t
=f^+\circ \Phi_t\, \chi_{\Omega^+}
+f^-\circ \Phi_t\, \chi_{\Omega^-}.
\end{align*}
Thus we introduce the function $\tilde f_t := f^+\circ \Phi_t\, \chi_{\Omega^+} +f^-\circ \Phi_t\, \chi_{\Omega^-}$.
Now we  obtain the canonical form  \eqref{G_lag} for the Lagrangian:
	\begin{align}
	\label{G_Omega}
	G(t,\varbf,\psib) & = a(t,\varbf,\psib) + b(t,\varbf),
	\end{align}
	with
	\begin{align*}
	a(t,\varbf,\psib) :=& \int_D \sigma A(t)\nabla \phdi\cdot\nabla \psdi
	-\tilde f_t \psdi \xi(t)  - \int_{\Gamma_n} \sigma^{-1} \partial_n\psdi(\phdi - h) \\
	\notag & + \int_D \sigma A(t)\nabla \phne\cdot\nabla \psne 
	-\tilde f_t  \psne \xi(t)  - \int_{\Gamma_n} g\psne,\\
	b(t,\varbf) :=&  \frac{\alpha_1}{2}\int_D (\phdi - \phne)^2 \xi(t) + \frac{\alpha_2}{2}\int_{\Gamma_n} (\phne - h)^2
	\end{align*}
	where the Jacobian $\xi(t)$ and $A(t)$ are defined as
	$
	\xi(t)  := \det (D\Phi_t)$ and  
	$A(t)  := \xi(t)D\Phi_t^{-1}D\Phi_t^{-T}$.
	In the previous expression \eqref{G_Omega}, one should note that the integrals on subsets of $\Gamma$ are unchanged since $\Phi_t^{-1} = I$ on $\Gamma$.  Thus we  have $\Phi_t^{\VV}(D) = D$, however the terms inside the integrals on $D$ are modified by the change of variable since $\Phi_t^{-1}\neq I$ inside $D$.
	Note that
	\begin{align*}
	J(\Omega^+(t))=G(t,\ub^t,\psib),\text{ for all } \psib\in H^1_0 (D)\times H^1_d (D), 
	\end{align*}
	where $\ub^t=(\une^t,\udi^t):=(u_{n,t}\circ\Phi_t,u_{d,t}\circ\Phi_t)$ and $u_{n,t},u_{d,t}$ solve \eqref{varun},\eqref{varud}, respectively, with the domains $\Omega^+$ and $\Omega^-$ replaced by $\Omega^+(t)$ and $\Omega^-(t)$. 
	As one can verify by applying a change of variables to \eqref{varun} and \eqref{varud} on the domain $\Omega^+(t)$ the functions $\une^t,\udi^t$ satisfy
	\begin{align}\label{varun_t}
	\int_D \sigma A(t)\nabla \une^t \cdot \nabla \hat{\psi}_n &= \int_D \tilde f_t\hat{\psi}_n + \int_{\Gamma_n} g\hat{\psi}_n \mbox{ for all } \hat{\psi}_n\in H^1_d (D), \\
	\label{varud_t}
	\int_D \sigma A(t) \nabla \udi^t \cdot \nabla \hat{\psi}_d &= \int_D \tilde f_t \hat{\psi}_d  \mbox{ for all } \hat{\psi}_d\in H^1_0(D) .
	\end{align}
	Applying standards estimates for elliptic partial differential equations and the fact that $A(t)$ is uniformly bounded from below and above for $t$ small enough, we infer from equations \eqref{varun_t},\eqref{varud_t} the existence of constants  $C_1,C_2>0$ independent of $t$ and $\tau>0$ such that for all $t\in [0,\tau]$:
	\ben\label{apriori_u_D_u_N}
	\|\udi^t\|_{H^1(D)}\le C_1,\quad \text{ and } \|\une^t\|_{H^1(D)}\le C_2.
	\een 
	From these estimates, we get
	$\udi^t\rightharpoonup w_d \text{ and } \une^t\rightharpoonup w_n \text{ in } H^1(D)\mbox{ as }t\to 0.$
	Passing to the limit in \eqref{varun_t} and \eqref{varud_t} yields  $w_d=\udi$ and $w_n=\une$ by uniqueness.
	
	Let us now check Assumption \ref{amp:gateaux_diffbar_G} and the conditions of Theorem~\ref{thm:sturm} for the function $G$ given by \eqref{G_Omega} and the Banach spaces $E = H^1_d (D)\times H^1_d (D)$ and $F = H^1_0 (D)\times H^1_d (D)$. First of all equation \eqref{eq:state_G} admits a unique solution $\ub^t := (\une^t,\udi^t)$ for each $t\in[0,\tau]$.
	The  conditions of Assumption \ref{amp:gateaux_diffbar_G} are readily satisfied and also the
	function $G$ is affine with respect to $\psi = (\psi_d,\psi_n)$.
	
	Regarding the conditions of Theorem~\ref{thm:sturm}, first note that applying   Lax-Milgram's lemma, we check that both equations \eqref{eq_bar_p_t_1} and \eqref{eq_bar_p_t_2} have indeed a unique solution in $F = H^1_0 (D)\times H^1_d (D)$: 
	\begin{align}
	\label{eq_bar_p_t_1}  &\int_D \sigma A(t)\nabla \pdi^t\cdot\nabla \hat{\varphi}_d +\frac{\alpha_1}{2}\int_D \xi(t) (\udi^t+\udi - (\une^t+\une)) \hat{\varphi}_d
	- \int_{\Gamma_n} \sigma^{-1}\partial_n \pdi^t\hat{\varphi}_d =0,\\ 
	\label{eq_bar_p_t_2} &\int_D \sigma A(t)\nabla \pne^t\cdot\nabla \hat{\varphi}_n-\frac{\alpha_1}{2}\int_D \xi(t)(\udi^t+\udi - (\une^t+\une)) \hat{\varphi}_n   + \frac{\alpha_2}{2}\int_{\Gamma_n} (\une^t+\une - 2h)\hat{\varphi}_n =0,
	\end{align}
	for all $\hat{\varphi}_d$, $\hat{\varphi}_n$ in $H^1_d (D)$. 
	Therefore there exists a unique solution $\pb^t= (p_n^t,p_d^t)$ of the averaged adjoint equation \eqref{averated_}.
	
	Now we check Assumption (H1). Testing  \eqref{eq_bar_p_t_1} with 
	$\hat{\varphi}_d=p_d^t$ and \eqref{eq_bar_p_t_2} with $\hat{\varphi}_n=p_n^t$, we conclude by an application of H\"older's inequality together with \eqref{apriori_u_D_u_N} the existence of constants $C_1,C_2$ and $\tau>0$ such that for all $t\in [0,\tau]$
	\benn
	\|p_d^t\|_{H^1(D)}\le C_1,\quad \text{ and } \|p_n^t\|_{H^1(D)}\le C_2.
	\eenn 
		We get that for each sequence $t_k$ converging to zero, there exists a subsequence also denoted $t_k$ such that 
		$p_d^{t_k}\rightharpoonup q_d$ and $p_n^{t_k}\rightharpoonup q_n$ for two elements $q_d,q_n\in H^1(D)$.
		Passing to the limit in 
		\eqref{eq_bar_p_t_1} and \eqref{eq_bar_p_t_2} yields  $q_d=\pdi$ and $q_n=\pne$ by uniqueness, where $\pdi$ and $\pne$ are solutions of the adjoint equations. Since the limit is unique, we have in fact $p_d^{t}\rightharpoonup \pdi$ and $p_n^{t}\rightharpoonup \pne$ as $t\to 0$.
	Finally, differentiating $G$ with respect to $t$ yields
	\begin{align*}
	\partial_tG(t,\varbf,\psib) & = \frac{\alpha_1}{2}\int_D (\phdi - \phne)^2  \xi(t)\tr(D\VV_t D \Phi_t^{-1})\\
	&\hspace{-1.5cm} + \int_D \sigma A'(t)\nabla \phdi\cdot\nabla \psdi -\tilde f_t \psdi \xi(t)\tr(D\VV_t D \Phi_t^{-1}) -\psdi  \widetilde{\nabla} f_t\cdot \VV_t  \xi(t)\\
	&\hspace{-1.5cm} + \int_D \sigma A'(t)\nabla \phne\cdot\nabla \psne - \tilde f_t\psne \xi(t)\tr(D\VV_t D \Phi_t^{-1}) -\psne \widetilde{\nabla} f_t\cdot \VV_t \xi(t).
	\end{align*} 
	where 
	$$\widetilde{\nabla} f_t := \nabla f^+\circ \Phi_t \, \chi_{\Omega^+} + \nabla f^-\circ \Phi_t\, \chi_{\Omega^-},\qquad \VV_t = \VV\circ\Phi_t$$
	$$A'(t) = \tr(D\VV^t D\Phi_t^{-1}) A(t) - D\Phi_t^{-T} D\VV_t A(t) -(D\Phi_t^{-T} D\VV_t A(t) )^T$$ 
	and $D\VV_t$ is the Jacobian matrix of $\VV_t$.
	In view of $\VV\in C^1_c(D,\R^d)$, the functions $t\mapsto D\VV_t$ and $t\mapsto \tr( D\VV_t \Phi_t^{-1})= \divv(\VV)\circ \Phi_t$ are continuous on $[0,T]$. Moreover $\phdi,\phne,\psdi,\psne$ are in $H^1(D)$, $f\in PH^1(D)$ so that $\partial_tG(t,\varbf,\psib) $ is well-defined for all $t\in [0,T]$. 
		Using the weak convergence of $\pb^t$ and the strong differentiability of 
		$t\mapsto A(t)$ and $t\mapsto \xi(t)$  it follows 
		\ben
		\lim_{t \searrow 0 }  \frac{G( t,\ub^0,\pb^t ) - G(0,\ub^0, \pb^t) }{t}=\partial_t G(0,\ub^0,\pb^0).
		\een
	
	Thus we have verified all assumptions from Theorem \ref{thm:sturm}. This yields
	\benn
	dJ(\Omega^+)(\VV)=\dt \left(G(t,\ub^t,\psib)\right)|_{t=0}=\partial_t G(0,\ub^0,\pb^0)\text{ for all } \psib\in F= H^1_0 (D)\times H^1_d (D) , 
	\eenn 
	and therefore we have proved the following result.
	\begin{proposition}[distributed shape derivative]\label{domexpr} 
		Let $D\subset\R^d$ be a Lipschitz domain,  $\VV\in C^1_c(D,\R^d)$, $f~\in~PH^1(D)$, $g\in H^{-1/2}(\Gamma_n)$, $h\in H^{1/2}(\Gamma_n)$, $\Omega^+\subset D$ is an open set, then the  shape derivative of $J(\Omega^+)$ is given by 
		\begin{align}\label{volexpr}
		\begin{split}
		dJ(\Omega^+)(\VV) & = \int_D \left(\frac{\alpha_1}{2}(\udi - \une)^2 -f(\pne+\pdi)\right)\divv \VV \\
		&\hspace{-1cm}+ \int_D  - (\pdi+\pne)\widetilde{\nabla} f\cdot \VV + \sigma A'(0)(\nabla \udi\cdot\nabla \pdi  + \nabla \une\cdot\nabla \pne)  ,
		\end{split}
		\end{align} 
		where $\widetilde{\nabla} f := \nabla f^+ \, \chi_{\Omega^+} + \nabla f^-\, \chi_{\Omega^-}$, $A'(0) = (\di \VV) I - D\VV^T - D\VV$,  $\une,\udi$ are solutions of  \eqref{varun},\eqref{varud} and  $\pne,\pdi$ of \eqref{eit5}, \eqref{eq:lambda}. 
		
		The shape derivative \eqref{volexpr} also has the tensor representation corresponding to \eqref{ea:volume_from}:
		\begin{align}\label{volexpr2}
		\begin{split}
		dJ(\Omega^+)(\VV) & = \int_D \Sb_1 \cdot D\VV + \Sb_0\cdot \VV  ,
		\end{split}
		\end{align} 
		where 
		\begin{align*}
		\Sb_1 & = -\sigma (\nabla \udi \otimes \nabla \pdi + \nabla \pdi \otimes \nabla \udi + \nabla \une \otimes \nabla \pne + \nabla \pne \otimes \nabla u_n)   \\
		&\quad  +\sigma(\nabla \udi\cdot\nabla \pdi  + \nabla \une\cdot\nabla \pne)I + \left(\frac{\alpha_1}{2}(\udi - \une)^2 -f(\pne+\pdi)\right)I,\\
		\Sb_0 & = - (\pdi+\pne)\widetilde{\nabla} f.
		\end{align*}
	\end{proposition}
	Note that the volume expression of the shape gradient in  Proposition \ref{domexpr}  has been obtained without any regularity assumption on $\Omega^+$. In order to obtain a boundary expression on the interface $\Gamma^+$  we need more regularity of  $\Omega^+$ provided by Assumption \ref{assump1}. If it is satisfied,  we can apply Corollary \ref{corollary2} to obtain directly the boundary expression of the shape derivative, using mainly the standard tensor relation $(\nabla \udi \otimes \nabla \pdi) n = (\nabla \pdi\cdot n)\nabla \udi$, which yields Proposition \ref{prop2}.
	%
	\begin{proposition}[boundary expression]\label{prop2}
		Under Assumption \ref{assump1} and $\VV\in C^1_c(D,\R^d)$ the shape derivative of $J(\Omega^+)$ is given by
   \begin{align*}
   dJ(\Omega^+)(\VV)  =&  \int_{\Gamma^+} \left[ \sigma(-\dn \udi\dn \pdi - \dn \une\dn \pne )\right]_{\Gamma^+} \VV\cdot n
   \\
   & + \int_{\Gamma^+}  ( [\sigma]_{\Gamma^+}(\nabla_{\Gamma^+} \udi\cdot\nabla_{\Gamma^+} \pdi + \nabla_{\Gamma^+} \une\cdot \nabla_{\Gamma^+} \pne ) -[f]_{\Gamma^+}(\pne+\pdi)) \VV\cdot n.
   \end{align*}
	\end{proposition}
	Note that our results cover and generalize several results that can be found in the literature of shape optimization approaches for EIT, including \cite{MR2329288,MR2536481}. For instance taking $\alpha_2 = 1$, $\alpha_1 = 0$ in Proposition \ref{prop2} we get $p_D \equiv 0$ which yields the same formula
 as the one obtained in \cite[pp. 533]{MR2329288}. 
	
	Note also that from a numerical point of view, the boundary expression in Proposition \ref{prop2} is delicate to compute compared to the domain expression in Proposition \ref{domexpr} for which the gradients of the state and adjoint states can be straightforwardly computed at grid points when using the finite element method for instance. The boundary expression, on the other hand, needs here the computation of the normal vector and the interpolation of the gradients on the interface $\Gamma^+$ which requires a precise description of the boundary and introduces an additional error.
	\section{Level set method}\label{section5}
	The level set method, originally introduced in \cite{MR965860}, gives a general framework for the computation of evolving interfaces using an implicit representation of these interfaces.  The core idea of this method is to represent the boundary of the moving domain $\Omega^+(t)\subset D \in \R^N$ as the level set of a continuous function $\phi(\cdot,t): D\to\R$.
	
	Let us consider the family of domains $\Omega^+(t)\subset D$ as defined in \eqref{domain}. Each domain $\Omega^+(t)$ can be defined as
	\begin{equation}
	\Omega^+(t) :=\{x\in D,\ \phi(x,t) < 0\}
	\end{equation}
	where $\phi: D\times \R^+ \to \R$ is continuous and called {\it level set function}. Indeed, if we assume $|\nabla\phi(\cdot,t)|\neq 0$ on the set  $\{x\in D,\ \phi(x,t) = 0\}$ then we have
	\begin{equation}
	\partial \Omega^+(t) = \{x\in D,\ \phi(x,t) = 0\},
	\end{equation}
	i.e. the boundary $\partial \Omega^+(t)$ is the zero level set of  $\phi(\cdot,t)$.

	Let $x(t)$ be the position of a particle on the boundary
	$\partial \Omega^+(t)$ moving with velocity $\dot{x}(t)=\VV(x(t))$ according to \eqref{Vxt}.
	Differentiating the relation $\phi(x(t),t)=0$ with respect to $t$
	yields the Hamilton-Jacobi equation:
	\begin{equation*}
	\partial_t\phi (x(t),t)+ \VV(x(t))\cdot \nabla \phi(x(t),t) = 0  \quad \mbox{ in } \partial \Omega^+(t)\times\R^+,
	\end{equation*}
	which is then extended to all of $D$ via the equation
	\begin{equation}
	\partial_t\phi(x,t) + \VV(x)\cdot \nabla \phi(x,t) = 0  \quad \mbox{ in } D\times\R^+,
	\label{eq:transport}
	\end{equation}
	or alternatively to $U\times\R^+$ where $U$ is a neighbourhood of $\partial \Omega^+(t)$.
	
	Traditionally, the level set method has been designed to track smooth interfaces moving along the normal direction to the boundary. Theoretically, this is supported by Theorem \ref{thm:structure_theorem}, i.e. if the domain $\Omega^+(t)$ and the shape gradient are smooth enough then the shape derivative only depends on $\VV\cdot n$ on $\partial\Omega^+(t)$. In this case, we may choose for the optimization a vector field $\VV = \vartheta_n n$ on $\partial\Omega^+(t)$. Then, noting that an extension to $D$ of the unit outward normal vector $n$ to $\Omega^+(t)$ is given 
	by $n =\nabla \phi /|\nabla \phi|$, and extending $\vartheta_n$ to all of $D$, one obtains from \eqref{eq:transport} the level set equation 
	\begin{equation}
	\partial_t \phi + \vartheta_n |\nabla \phi| = 0 \quad \mbox{ in } D\times\R^+.
	\label{eq:HJ1}
	\end{equation}
	The initial data $\phi(x,0)=\phi_0(x)$ accompanying the Hamilton-Jacobi equation \eqref{eq:transport} or \eqref{eq:HJ1} is chosen as the signed distance function to the
	initial boundary $\partial \Omega^+(0)$ in order to satisfy the condition $|\nabla u|\neq 0$ on $\partial\Omega^+$, i.e.
	\begin{equation}
	\phi_0(x) = 
	\left\{\begin{array}{rl}
	d(x,\partial \Omega^+(0)), & \mbox{ if } x\in (\Omega^+(0))^c, \\
	-d(x,\partial \Omega^+(0)), & \mbox{ if } x\in \Omega^+(0) .
	\end{array}
	\right.
	\end{equation}

	\subsection{Level set method and domain expression}
	
	In the case of the distributed shape derivative, for instance  \eqref{volexpr} or \eqref{volexpr2}, $\phi$ is not governed by \eqref{eq:HJ1} but rather by the Hamilton-Jacobi equation \eqref{eq:transport}. Indeed we obtain a descent direction $\VV$ defined in $D$ by solving \eqref{VP_1}, where $dJ(\Omega^+)$ is given by Proposition \ref{domexpr} which can subsequently be used in \eqref{eq:transport} to compute the evolution of $\phi$. On the other hand, in the usual level set method, one solves a PDE {\it on the  boundary} $\partial\Omega^+$ in an analogous way as for   \eqref{VP_1} (for instance using a Laplace-Beltrami operator), and uses the boundary expression from Proposition \ref{prop2} to obtain $\vartheta_n = \VV\cdot n$ on $\partial\Omega^+$. 
	
	Numerically it is actually more straightforward in many cases to use \eqref{eq:transport} instead of   \eqref{eq:HJ1}. Indeed, when using \eqref{eq:HJ1}, $\vartheta_n$ is initially only given on $\partial \Omega^+(t)$ and must be extended to the entire domain $D$ or at least to a narrow band around $\partial \Omega^+(t)$. 
	Therefore it is convenient to use \eqref{eq:transport} with  $\VV$  already defined in $D$ as is the case of the distributed shape derivative, which provides an extension to $D$ or to a narrow band around $\partial \Omega^+(t)$.
	
	In shape optimization,  $\vartheta_n$ usually depends on the solution of one or several PDEs and their gradient. Since the boundary $\partial\Omega^+(t)$  in general does not match the grid nodes where $\phi$ and the solutions of the  partial differential equations are defined in the numerical application, the computation of $\vartheta_n$ 
	requires the  interpolation on $\partial\Omega^+(t)$ of functions defined at the grid points only, complicating the numerical implementation  and introducing an 
	additional interpolation error. This is an issue in particular for interface problems where $\vartheta_n$ is the jump of a function across the interface, as in Proposition \ref{prop2}, which requires multiple interpolations and is error-prone. In the distributed shape derivative framework  $\VV$ only needs to be defined at grid nodes.

	\subsection{Discretization of the level set equation}
	Let $D$ be the unit square $D=(0,1)\times(0,1)$ to fix ideas.
	For the discretization of the Hamilton-Jacobi equation \eqref{eq:transport},
	we first define the mesh grid corresponding to $D$. We introduce the nodes $P_{ij}$
	whose coordinates are given by $(i\Delta x, j \Delta y)$, $1\leq i,j\leq N$ where $\Delta x$ 
	and $\Delta y$ are the steps discretization in the $x$ and $y$
	directions respectively. Let us also write $t^k = k\Delta t$
	the discrete time for $k\in \mathbb{N}$, where $\Delta t$ is the time
	step. We are seeking for an approximation
	$\phi_{ij}^k \simeq \phi(P_{ij},t^k)$.
	
	In the usual level set method, the level set equation \eqref{eq:HJ1} is discretized using an explicit upwind scheme proposed by Osher and Sethian \cite{MR1939127,MR965860,MR1700751}. This scheme applies to the specific form \eqref{eq:HJ1} but is not suited to discretize \eqref{eq:transport} required for our application. Equation \eqref{eq:transport} is of the form
	\begin{equation}
	\partial_t\phi + H(\nabla \phi)= 0  \quad \mbox{ in } D\times\R^+.
	\end{equation}
	where $ H(\nabla \phi) := \VV\cdot \nabla \phi $ is the so-called Hamiltonian. We use the Local Lax-Friedrichs flux originally conceived in \cite{MR1111446} and which reduces in our case to:
	$$\hat H^{LLF}(p^-,p^+,q^-,q^+)  = H\left(\frac{p^- + p^+}{2},\frac{q^- +q^+}{2}\right) -\frac{1}{2} (p^+-p^-)\alpha^x -\frac{1}{2}(p^+-p^-)\alpha^y  $$
	where $\alpha^x = |\VV_x|$, $\alpha^y = |\VV_y|$, $\VV = (\VV_x,\VV_y)$ and
	\begin{align*}
	p^- &= D_x^{-}\phi_{ij}=\dfrac{\phi_{ij}-\phi_{i-1,j}}{\Delta x}, &
	p^+  = D_x^{+}\phi_{ij}=\dfrac{\phi_{i+1,j}-\phi_{ij}}{\Delta x},\\
	q^- &= D_y^{-}\phi_{ij}=\dfrac{\phi_{ij}-\phi_{i,j-1}}{\Delta y}, &
	q^+  = D_y^{+}\phi_{ij}=\dfrac{\phi_{i,j+1}-\phi_{ij}}{\Delta y}
	\end{align*}
	are the backward and forward approximations of the $x$-derivative and $y$-derivative of 
	$\phi$ at $P_{ij}$, respectively. Using a forward Euler time discretization, the numerical scheme corresponding to \eqref{eq:transport} is 
	\begin{equation}
	\phi_{ij}^{k+1}=\phi_{ij}^k - \Delta t \ \hat H^{LLF}(p^-,p^+,q^-,q^+)
	\end{equation}
	

	For numerical accuracy, the solution of the level set equation 
	\eqref{eq:transport} should not be too flat or too steep. 
	This is fulfilled for instance if $\phi$ is the distance function
	i.e. $|\nabla \phi| = 1$. Even if one initializes $\phi$ using a
	signed distance function, the solution
	$\phi$ of the level set equation \eqref{eq:transport} does not generally
	remain close to a distance function. We may occasionally perform a reinitialization
	of $\phi$ by solving a parabolic equation up to the
	stationary state; see \cite{MR2356899,MR2459656,MR1723321}.
	Although in the level set method this reinitialization is standard, in the case of the distributed shape gradient, we observe experimentally that the level set function $\phi$ stays close to a distance function during the iterations and we do not need to reinitialize. 
		The regularization of the shape gradient could explain this observed stability of the level set function.
	
	The computational efficiency of the level set method can be improved by using the so-called ``narrow band'' approach introduced in \cite{MR1329634}, which consists in computing and updating the level set function only on a thin region around the interface. This allows to reduce the complexity of the problem to $N\log(N)$ instead of $N^2$ in two dimensions.
	In this paper we do not implement this approach but we mention that it could also be applied to the distributed shape derivative approach and equation \eqref{eq:transport} by taking $\VV$ with a support in a narrow band around the moving interface, which can be achieved by choosing the appropriate space $\Eb$ in \eqref{eq:bilinear}.

	\section{Application and numerical results}\label{section6}

	\subsection{Electrical impedance tomography}
	
	In this section we give numerical results for the problem of electrical impedance tomography presented in Section \ref{sec:eit}, precisely we look for an approximate solution of the shape optimization problem \eqref{EIT-SO}.
	Using the notations of Section  \ref{sec:eit} we take  $D = (0,1)\times (0,1)$ and $\Gamma_d = \emptyset$, i.e. we have measurements on the entire boundary $\Gamma$.  For easiness of implementation, we consider a slightly different problem than the one in Section \ref{sec:eit}. Denote $\Gamma_{t}$, $\Gamma_{b}$, $\Gamma_{l}$ and $\Gamma_{r}$ the four sides of the square, where the indices $t,b,l,r$ stands for top, bottom, left and right, respectively. We consider the following problems: find $\une\in H^1_{tb}(D)$
	\begin{equation}\label{varun-b}
	\int_D \sigma \nabla \une \cdot \nabla \varphi = \int_D f\varphi + \int_{\Gamma_{l}\cup\Gamma_{r}} g\varphi\ \mbox{ for all }\ \varphi\in H^1_{0,tb}(D) 
	\end{equation}
	and find $\udi\in H^1_{lr}(D)$ such that 
	\begin{equation}\label{varud-b}
	\int_D \sigma \nabla \udi \cdot \nabla \varphi = \int_D f\varphi + \int_{\Gamma_{t}\cup\Gamma_{b}} g\varphi\ \mbox{ for all }\ \varphi\in H^1_{0,lr}(D)  
	\end{equation}
	where
	\begin{align*}
	H^1_{tb}(D)&:=\{v\in H^1(D)\ |\ \; v=h\mbox{ on }\Gamma_t \cup \Gamma_b\}, \\
	H^1_{lr}(D)&:=\{v\in H^1(D)\ |\ \; v=h\mbox{ on }\Gamma_l \cup \Gamma_r\}, \\
	H^1_{0,tb}(D)&:=\{v\in H^1(D)\ |\ \; v=0\mbox{ on }\Gamma_t \cup \Gamma_b\}, \\
	H^1_{0,lr}(D)&:=\{v\in H^1(D)\ |\ \; v=0\mbox{ on }\Gamma_l \cup \Gamma_r\}.
	\end{align*}
	In our experiments we choose $f\equiv 0$. The results of Section \ref{sec:eit} can be straightforwardly adapted to equations \eqref{varun-b}, \eqref{varud-b}.
	
	We use the software package FEniCS for the implementation; see \cite{fenics:book}. The domain $D$ is meshed using a regular grid of $128\times 128$ elements and we describe the evolution of the interface $\Gamma^+$ using the level set method from Section \ref{section5}. The conductivity values are set to $\sigma_0 = 1$ and $\sigma_1 = 10$. 
	
	We obtain measurements $h_i$ corresponding to fluxes $g_i$, $i=1,..,I$, by taking the trace on $\Gamma$ of the solution of a Neumann problem where the fluxes are equal to $g_i$. To simulate real noisy EIT data, the measurements $h_i$ are corrupted by adding a normal Gaussian noise with mean zero and standard deviation $\delta*\|h_i\|_\infty$, where $\delta$ is a parameter. The noise level is computed as
	\begin{equation}\label{noise_compute}
	noise = \frac{\sum_{i=1}^I\|h_i - \tilde h_i\|_{L^2(\Gamma)}}{\sum_{i=1}^I\|h_i\|_{L^2(\Gamma)}} 
	\end{equation}
	where $\tilde h_i$ is the noisy measurement and $h_i$ the synthetic measurement without noise on $\Gamma$.

	We use a variation of the functional \eqref{eit3.1}, i.e. in our context:
	\begin{align}\label{funcJ_num}
	J(\Omega^+) &=  \sum_{i=1}^I \mu_i  \int_\Omega   \frac{1}{2} (u_{d,i}(\Omega^+) - u_{n,i}(\Omega^+))^2 ,
	\end{align}
	where $u_{d,i}$ and $u_{n,i}$ correspond to the different fluxes $g_i$.
	Here the coefficients $\mu_i$ are weights associated to the fluxes $g_i$. In our experiments we choose the weights $\mu_i$ such that each term of the sums in \eqref{funcJ_num} are equal to $1$ on initialization in order to have a well-distributed influence of each term. Practically, the $\mu_i$ are thus calculated during the first iteration.
	We use the distributed shape derivative $dJ(\Omega^+)$ from Proposition \eqref{domexpr}. 
	We obtain a descent direction by solving \eqref{VP_1} with $\Eb$ a finite dimensional subspace of $H^1_0(D)$ and
		$\mathcal B(v,w) = \int_D D v \cdot D w.$
		We choose $\Eb$ to be  the space of linear Lagrange elements.

	Since we use a gradient-based method we implement an Armijo line search to adjust the time-stepping. The algorithm is stopped when the decrease of the functional becomes insignificant, practically when the following stopping criterion is repeatedly satisfied: 
	$$J(\Omega^+_{k}) - J(\Omega^+_{k+1}) < \gamma (J(\Omega^+_{0}) - J(\Omega^+_{1})) $$  
	where $\Omega^+_{k}$ denotes the $k$-th iterate of $\Omega^+$. We take $\gamma = 5.10^{-5}$ in our tests.

	
	\begin{figure}[ht]
		\begin{center}
			\includegraphics[width=0.32\textwidth]{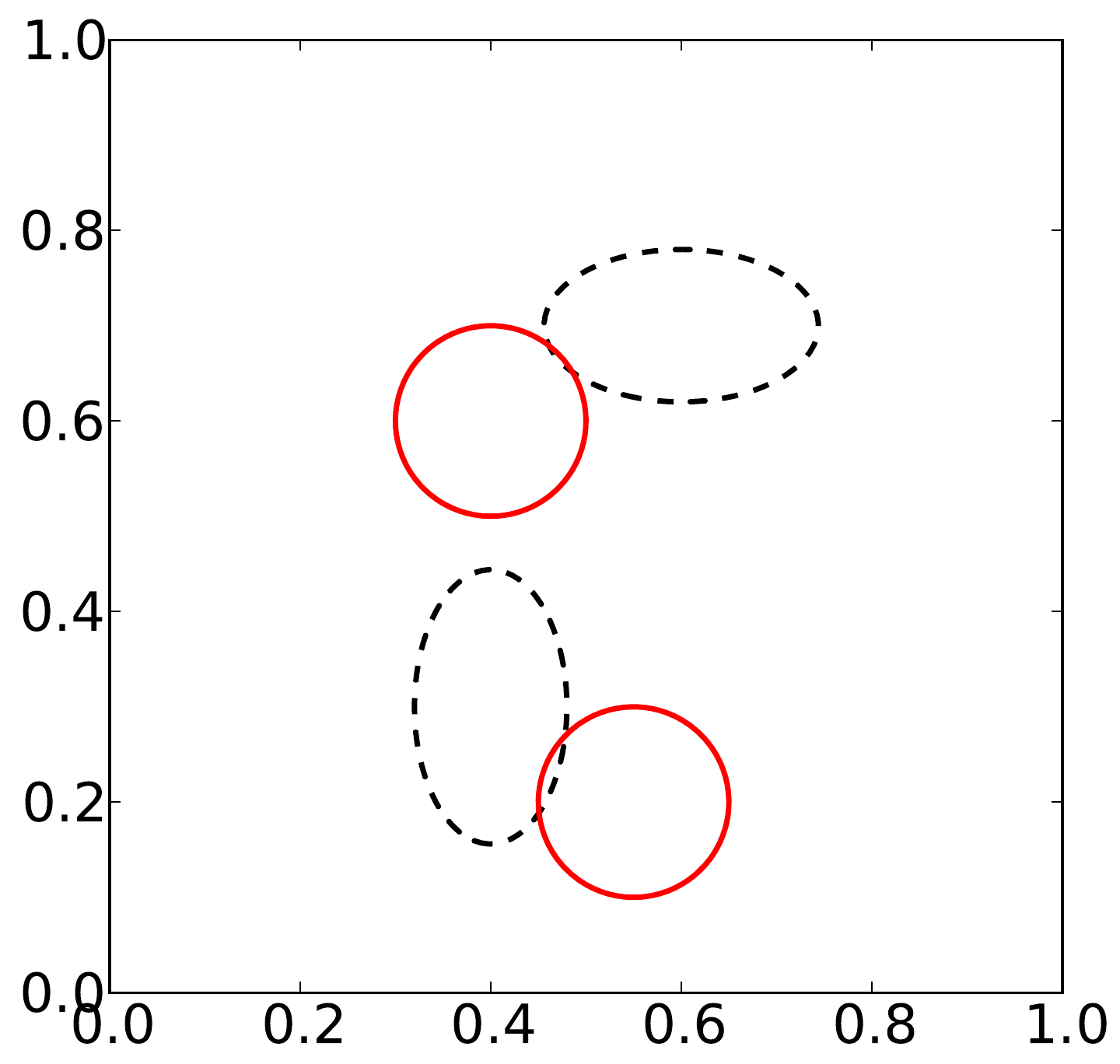}
			\includegraphics[width=0.32\textwidth]{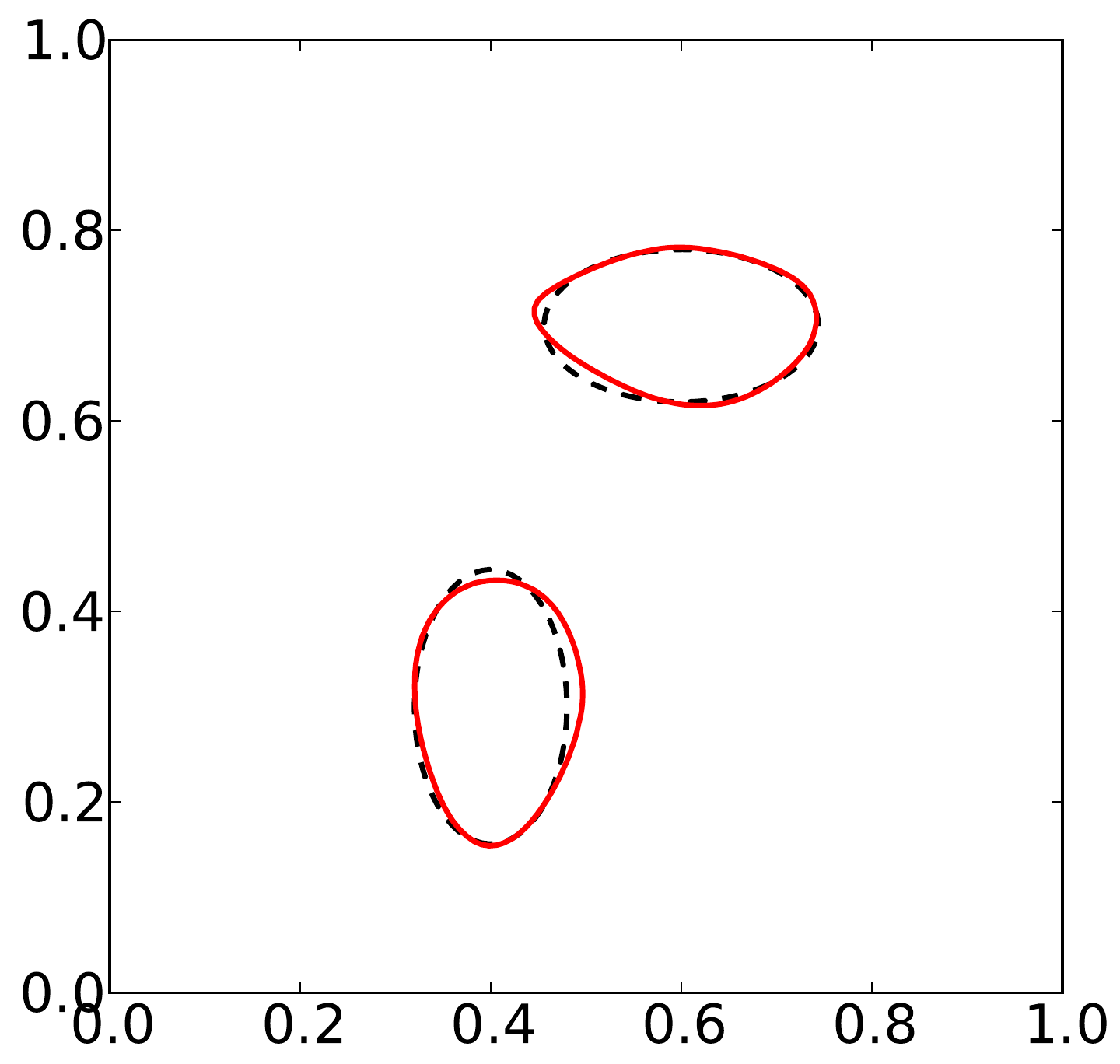}
			\includegraphics[width=0.32\textwidth]{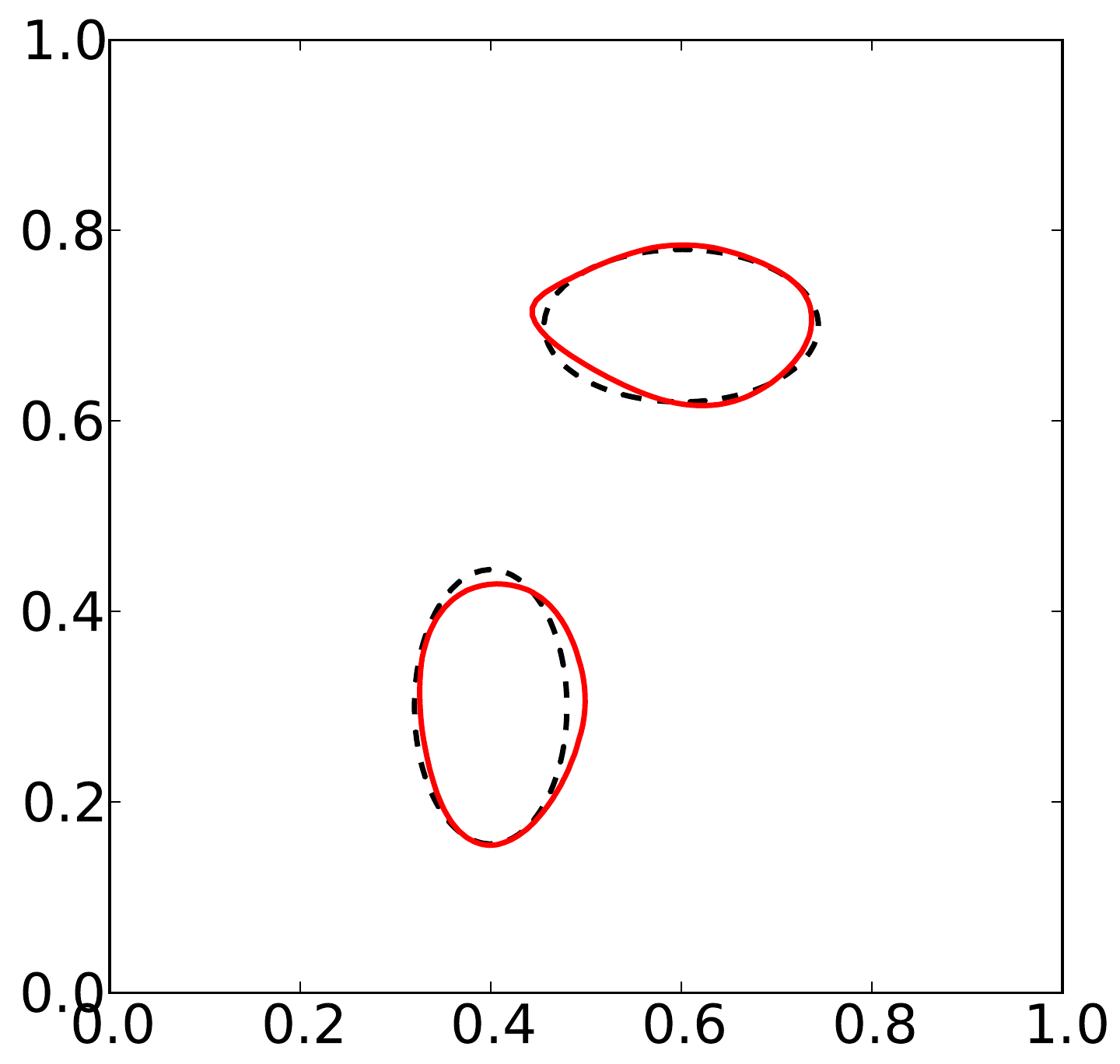}\\
			\includegraphics[width=0.32\textwidth]{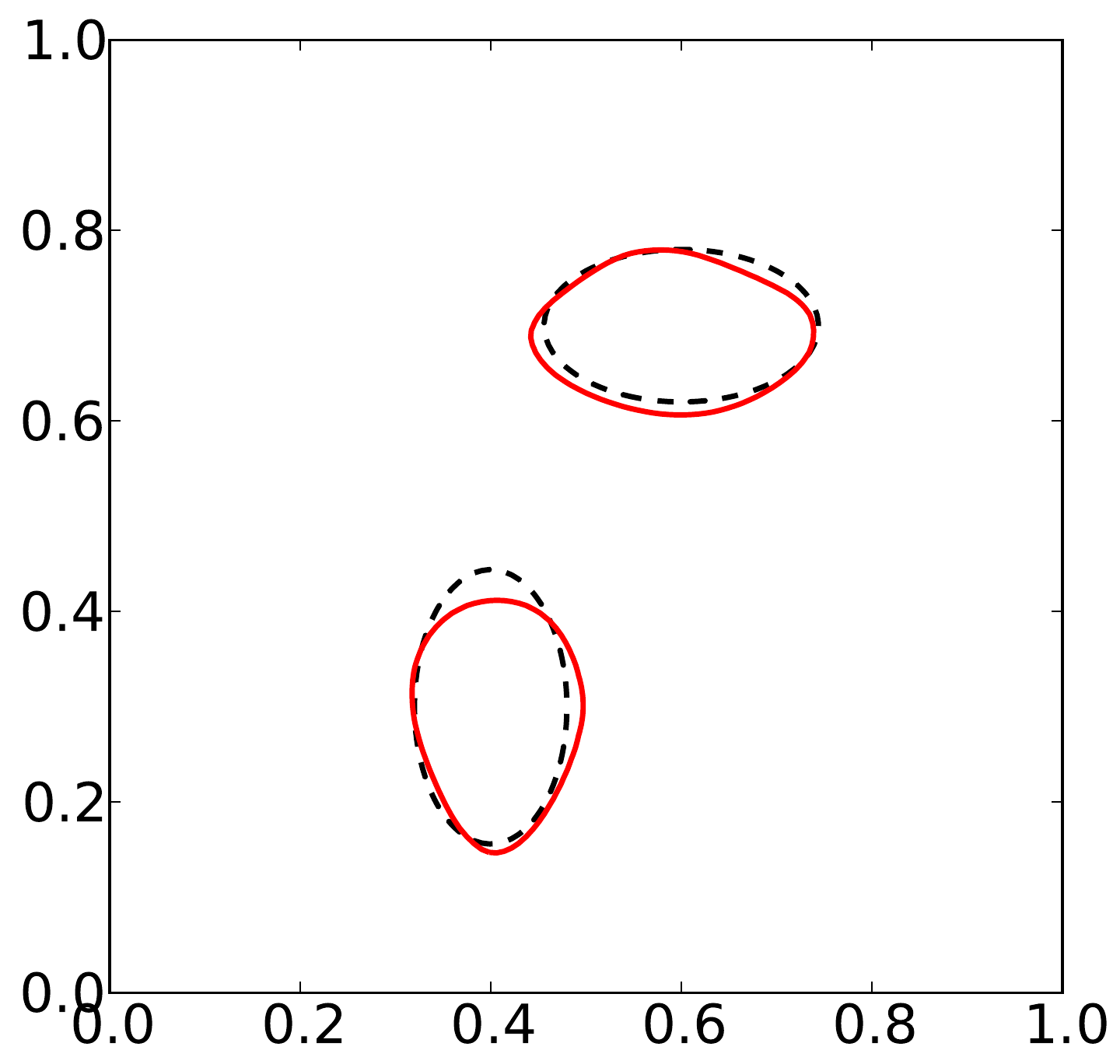}
			\includegraphics[width=0.32\textwidth]{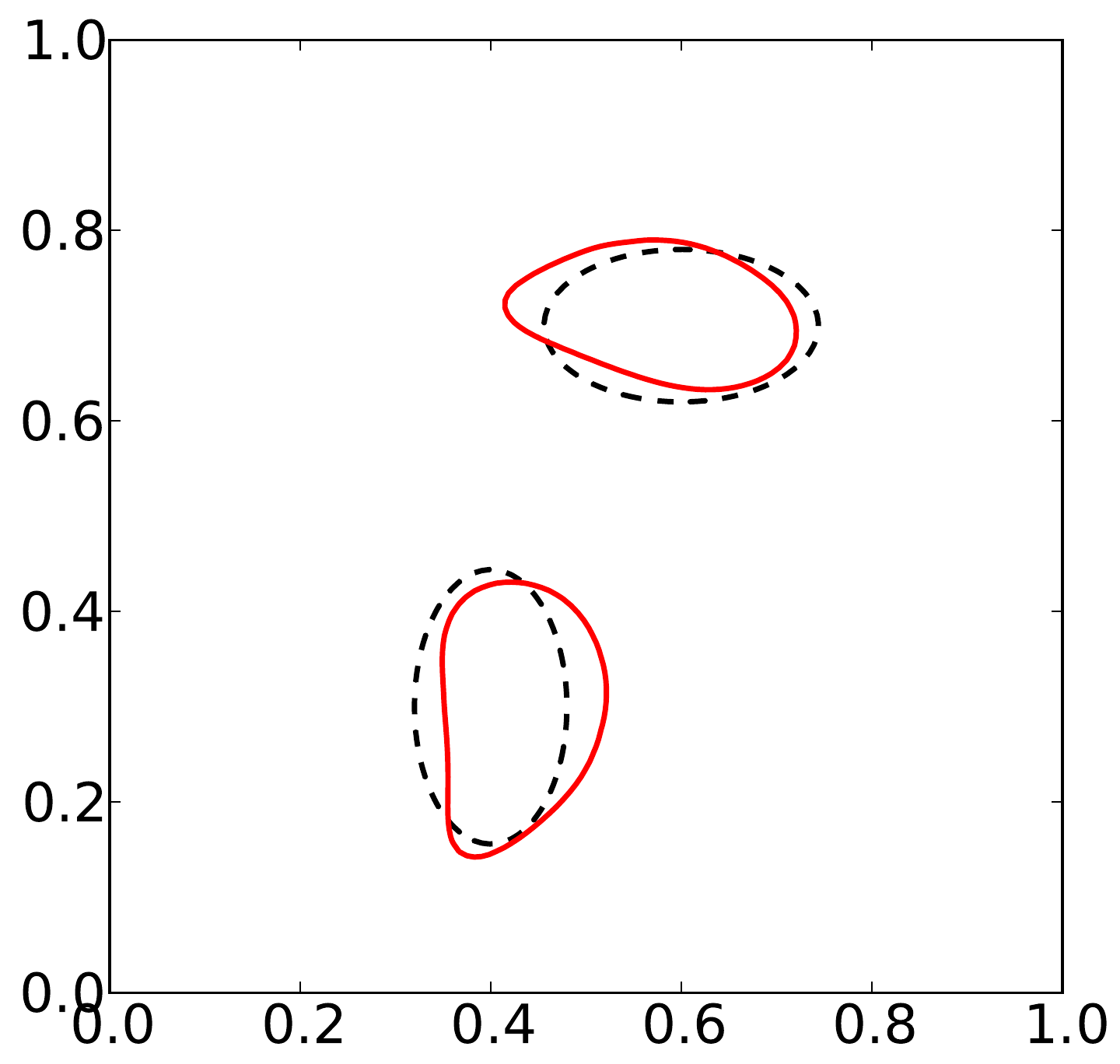}
			\includegraphics[width=0.32\textwidth]{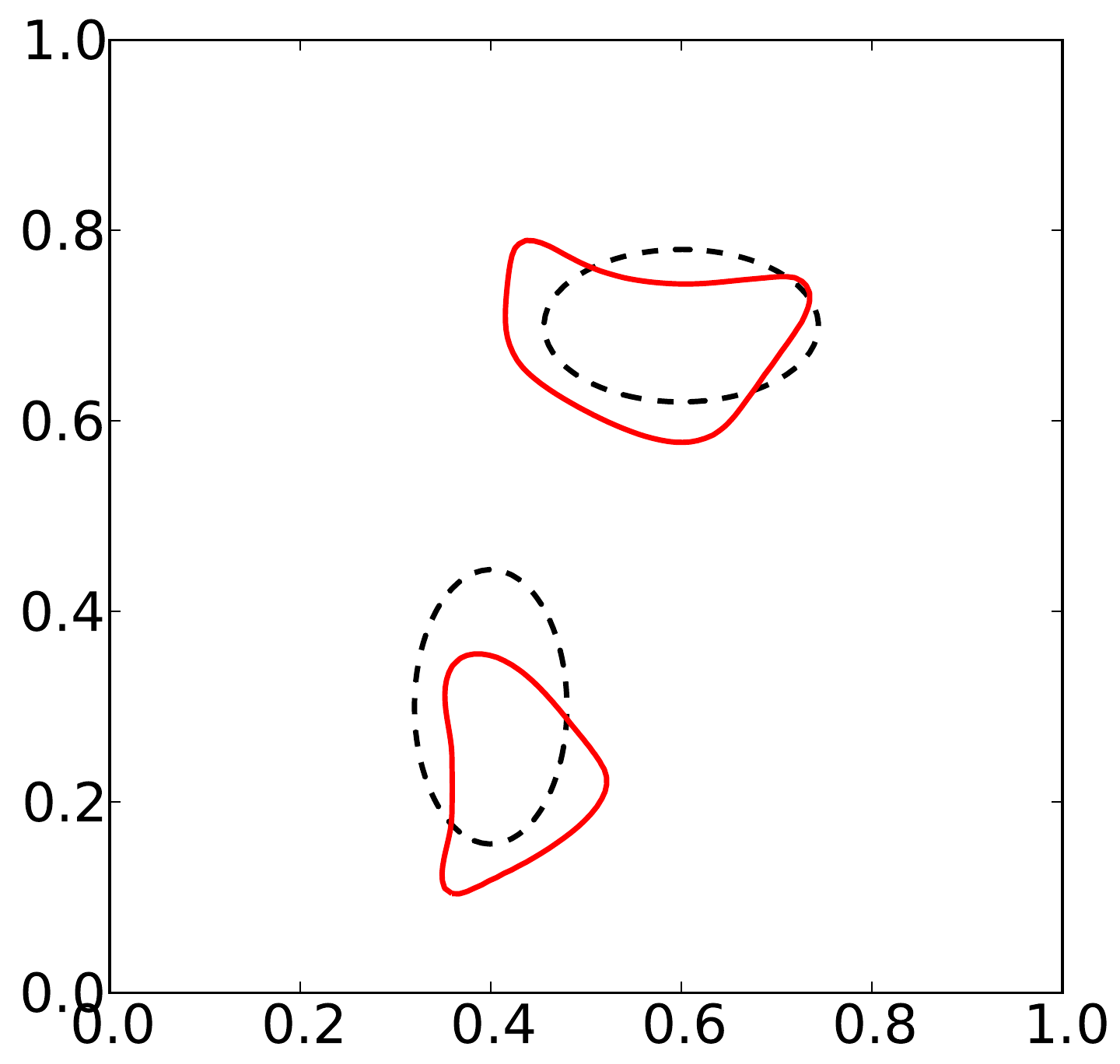}
		\end{center}
		\caption{Reconstruction (continuous contours) of two ellipses (dashed contours) with different noise levels and using three measurements. From left to right and top to bottom: initialization (continuous contours - top left), $0\%$ noise (367 iterations), $0.43\%$ noise (338 iterations), $1.44\%$ noise (334 iterations), $2.83\%$ noise (310 iterations), $7\%$ noise (356 iterations).}\label{rec_twoellipses_noise}
	\end{figure}
	
	In Figure \ref{rec_twoellipses_noise} we compare the reconstruction for different noise levels computed using \eqref{noise_compute}. We take in this example $I=3$, i.e. we use three fluxes $g_i$, $i=1,2,3$, defined as follows:
	\begin{align*}
	g_1 &= 1\mbox{ on }\Gamma_l\cup\Gamma_r \mbox{ and } g_1 = -1\mbox{ on }\Gamma_t\cup\Gamma_b, \\ 
	g_2 &= 1\mbox{ on }\Gamma_l\cup\Gamma_t \mbox{ and } g_2 = -1\mbox{ on }\Gamma_r\cup\Gamma_b,\\
	g_3 &= 1\mbox{ on }\Gamma_l\cup\Gamma_b \mbox{ and } g_3 = -1\mbox{ on }\Gamma_r\cup\Gamma_t. 
	\end{align*}
	Without noise, the reconstruction is very close to the true object and  degrades as the measurements become increasingly noisy,   as is usually the case in EIT. However, the reconstruction is quite robust with respect to noise considering that the problem is severely ill-posed. We reconstruct two ellipses and initialize with two balls placed at the wrong location. The average number of iterations until convergence is around $340$ iterations.  
	
	In Figure \ref{rec_three_inclusions} we reconstruct three inclusions this time using $I=7$ different measurements, with $1.55\%$ noise. The reconstruction is close to the true inclusion and is a bit degraded due to the noise. Figure \ref{rec_three_inclusions2} shows the convergence history of the cost functional in log scale for this example. 
	
	Our algorithm gives good results in comparison to existing results in the literature using level set methods to solve the EIT problem.
	In \cite{MR2132313} the EIT problem has been treated numerically using a level set method, which is not based on the use of shape derivatives but on the differentiation of a smooth approximation of the Heaviside function to represent domains. 
	In \cite{MR2536481} the level set method using the boundary expression of the shape derivative is used based on equation \eqref{eq:HJ1}.
	
	Our algorithm converges fast in comparison to \cite{MR2132313,MR2536481}:
	convergence occurs after around 300 iterations. In \cite{MR2132313} convergence occurs between 200 iterations for one inclusion and up to 50000 iterations for two inclusions. 
	In \cite{MR2536481} convergence occurs after 2000 or 10000 iterations on two examples with three inclusions. 
	Concerning measurements we obtain good reconstruction of two inclusions with $I=3$ and three inclusions with $I=7$, while in \cite{MR2132313} sets of $4,12,28$ and $60$ measurements are used but usually $60$ measurements are required for complicated shapes such as two inclusions. In \cite{MR2536481}, $60$ measurements are used.
	Nevertheless, our results are not directly comparable since the conductivities are unknown in \cite{MR2132313, MR2536481}, which makes the inverse problem harder and might explain the slower convergence.
	Also, the reconstructed shapes are not the same  although the complexity of the unknown shapes is comparable since we also consider two and three inclusions as in \cite{MR2132313, MR2536481}.
	Only an exact comparison using the same problem, test case, initialization, noise level, number and type of measurements could allow to conclude.

	\begin{figure}[ht]
		\begin{center}
			\includegraphics[width=0.4\textwidth]{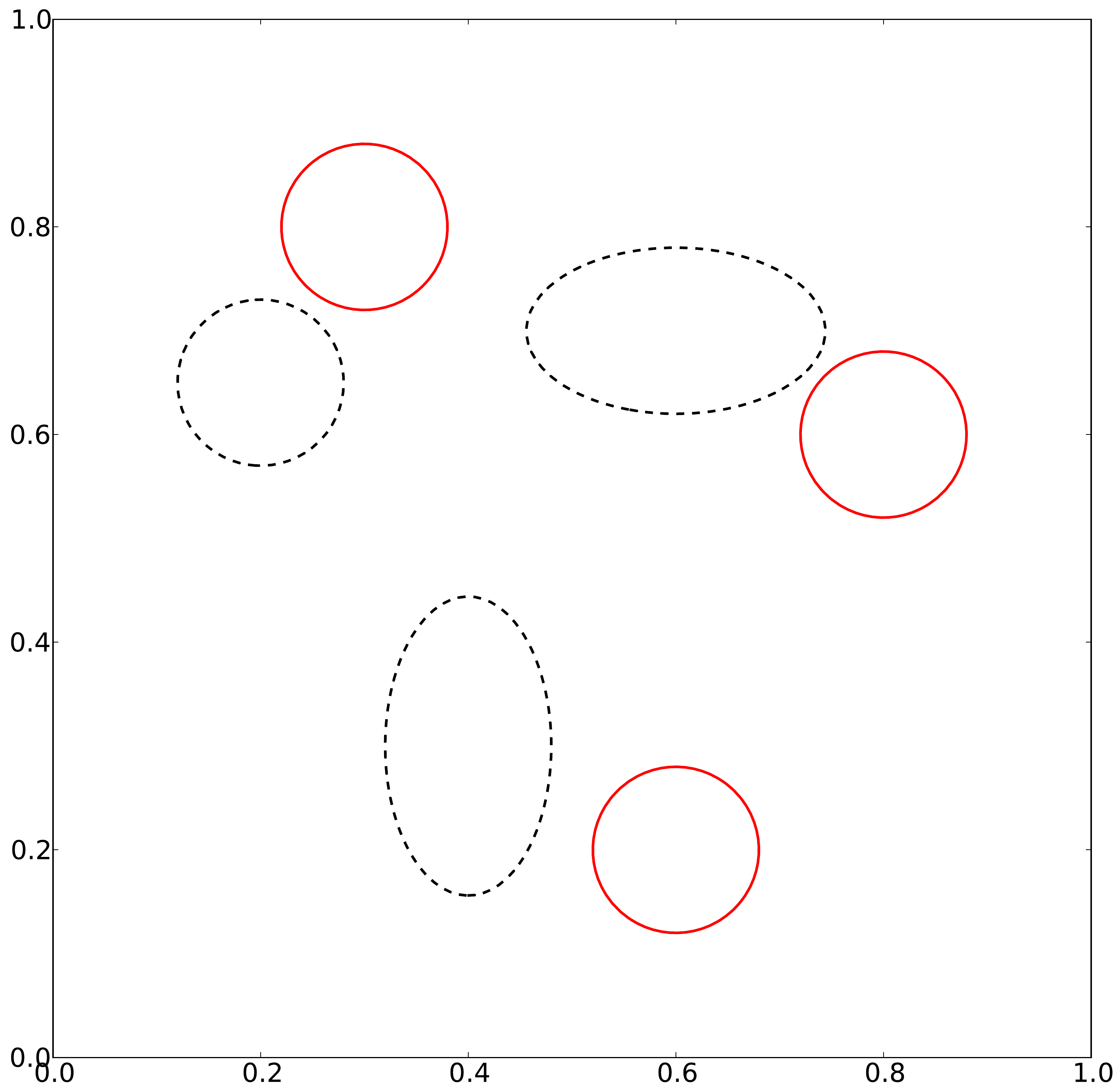}
			\includegraphics[width=0.4\textwidth]{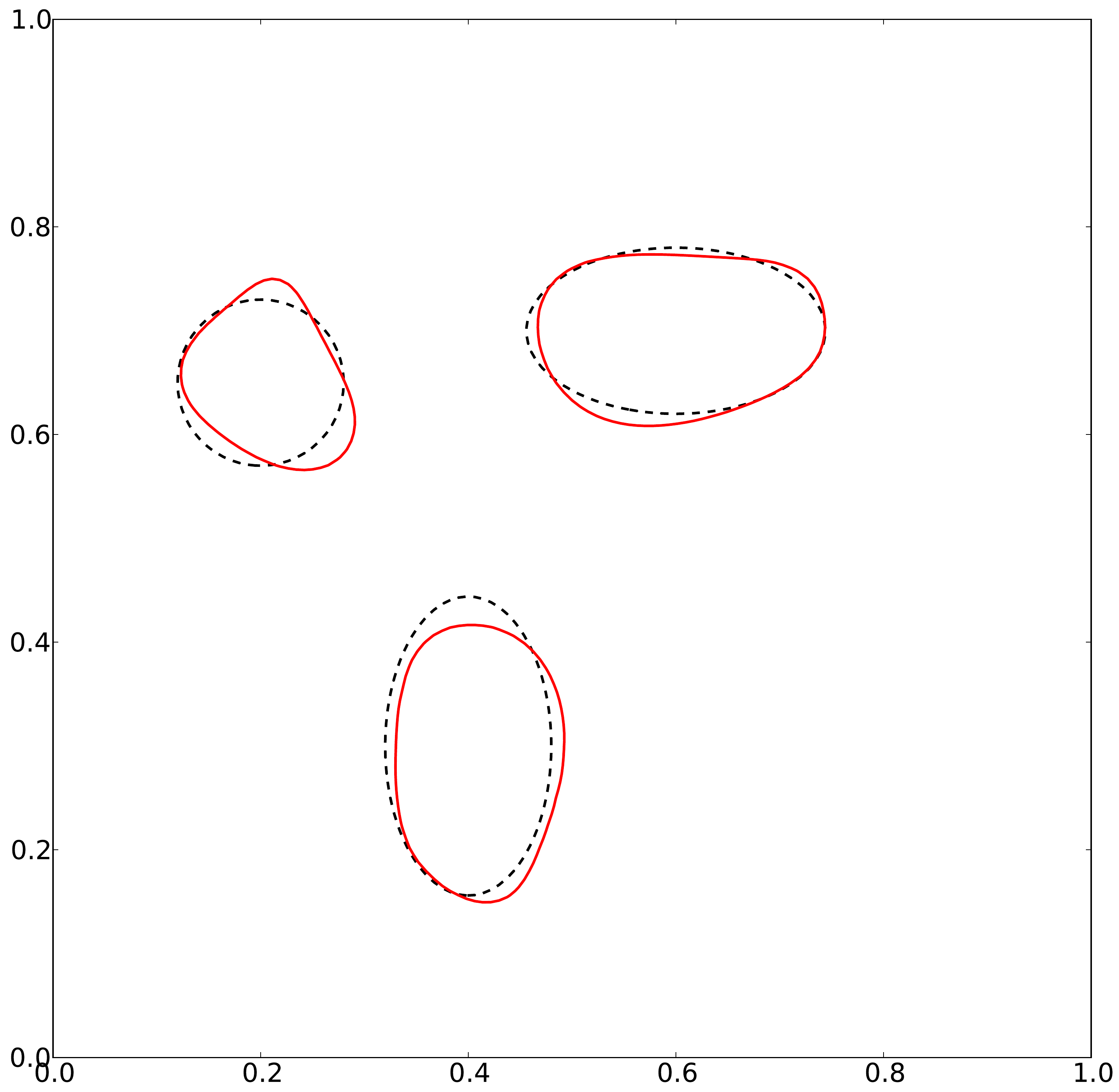}
		\end{center}
		\caption{Initialization (continuous contours - left) and reconstruction (continuous contours - right) of two ellipses and a ball (dashed contours) with $1.55\%$ noise (371 iterations) and using seven measurements.}\label{rec_three_inclusions}
	\end{figure}
	
	\begin{figure}[ht]
		\begin{center}
			\includegraphics[width=0.5\textwidth]{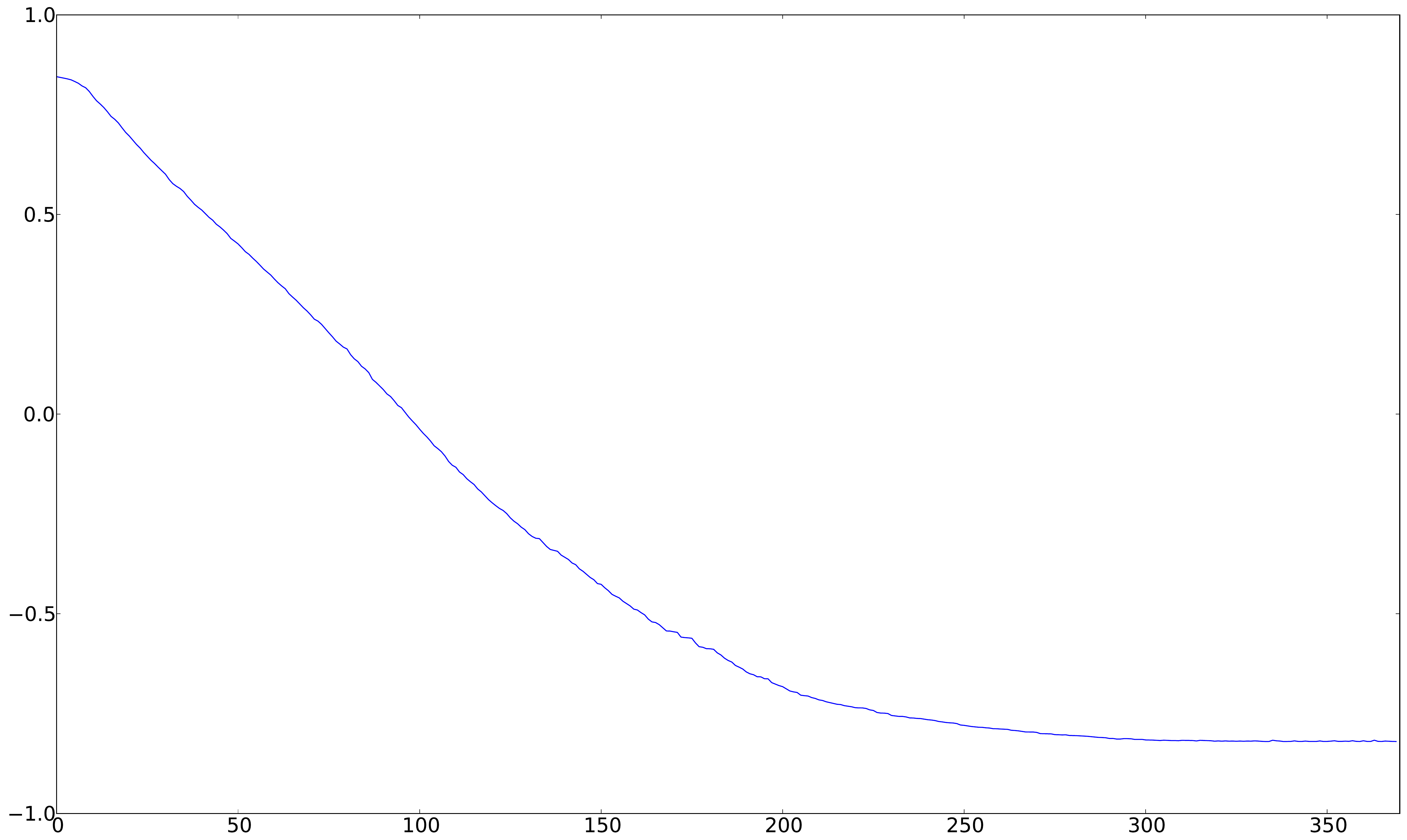}
		\end{center}
		\caption{History of cost functional corresponding to Figure \ref{rec_three_inclusions} in logarithmic scale.}\label{rec_three_inclusions2}
	\end{figure}

	{\bf Acknowledgments.} 
	The authors would like to thank the reviewers for their
	helpful comments. Antoine Laurain acknowledges support from the DFG research center MATHEON (MATHEON - Project C37, Shape/Topology optimization methods for inverse problems). Kevin Sturm acknowledges support from the  DFG research center MATHEON, Project C11.
	
	
	\bibliographystyle{abbrv}
	\bibliography{my_bib_paper4}

\end{document}